\documentclass[11pt]{article}

\usepackage[margin=1.05in]{geometry}
\usepackage{amsmath,amssymb,amsthm,mathtools}
\usepackage{booktabs}
\usepackage{graphicx}
\usepackage{float}
\usepackage{enumitem}
\usepackage{hyperref}
\usepackage{xcolor}
\usepackage{authblk}
\usepackage{aliascnt}
\hypersetup{colorlinks=true,linkcolor=blue!55!black,citecolor=blue!55!black,urlcolor=blue!55!black}

\newtheorem{theorem}{Theorem}[section]

\newaliascnt{lemma}{theorem}
\newtheorem{lemma}[lemma]{Lemma}
\aliascntresetthe{lemma}

\newaliascnt{proposition}{theorem}
\newtheorem{proposition}[proposition]{Proposition}
\aliascntresetthe{proposition}

\newaliascnt{corollary}{theorem}
\newtheorem{corollary}[corollary]{Corollary}
\aliascntresetthe{corollary}

\newaliascnt{conjecture}{theorem}

\aliascntresetthe{conjecture}

\theoremstyle{definition}
\newaliascnt{definition}{theorem}
\newtheorem{definition}[definition]{Definition}
\aliascntresetthe{definition}

\newaliascnt{remark}{theorem}
\newtheorem{remark}[remark]{Remark}
\aliascntresetthe{remark}

\usepackage[nameinlink,noabbrev]{cleveref}

\crefname{theorem}{Theorem}{Theorems}
\Crefname{theorem}{Theorem}{Theorems}
\crefname{lemma}{Lemma}{Lemmas}
\Crefname{lemma}{Lemma}{Lemmas}
\crefname{proposition}{Proposition}{Propositions}
\Crefname{proposition}{Proposition}{Propositions}
\crefname{corollary}{Corollary}{Corollaries}
\Crefname{corollary}{Corollary}{Corollaries}
\crefname{conjecture}{Conjecture}{Conjectures}
\Crefname{conjecture}{Conjecture}{Conjectures}
\crefname{definition}{Definition}{Definitions}
\Crefname{definition}{Definition}{Definitions}
\crefname{remark}{Remark}{Remarks}
\Crefname{remark}{Remark}{Remarks}

\newcommand{\Av}{\operatorname{Av}}
\newcommand{\A}{\mathcal A}

\newcommand{\spr}{\operatorname{spr}}
\newcommand{\last}{\operatorname{last}}

\title{Finite-state enumeration of adjacency-constrained\\132-avoiding permutations}
\author[1]{Teruki Mayama}
\author[1, *]{Dai Akita}
\date{}
\affil[1]{Department of Mechano-Informatics, Graduate School of Information Science and Technology, The University of Tokyo, Tokyo, Japan}
\affil[*]{d.akita@ne.t.u-tokyo.ac.jp}

\begin{document}
\maketitle

\begin{abstract}
For a fixed integer $m\ge 1$, let $\A_n^{(m)}$ be the set of permutations $\pi\in S_n$ that avoid the pattern $132$ and satisfy the adjacency bound $|\pi_{i+1}-\pi_i|\le m$ for all $i$. 
Here, a pattern $132$ means three indices $i<j<k$ such that $\pi_i<\pi_k<\pi_j$.  
A recent study initiated the enumeration of these constrained 132-avoiding permutations, treating the case \(m=2\) by deriving a rational ordinary generating function and asking for finite-state decompositions, rational generating functions, and explicit rational formulas for larger fixed \(m\).
We introduce a two-sided endpoint-state decomposition that works uniformly for every fixed $m$.  
The state variables impose threshold bounds on the endpoint deficiencies $n-\pi_1$ and $n-\pi_n$, with thresholds in $\{0,1,\ldots,m-1,\infty\}$. 
This gives at most \((m+1)^2\) states and proves that, for every fixed \(m\), the ordinary generating function \(A^{(m)}(x)\) is rational and can be computed effectively by exact linear algebra.
We also identify cyclic strongly connected components of the dependency graph in the finite-state system to give an explicit upper bound for the order of an eventual constant-coefficient recurrence satisfied by the sequence $a_n^{(m)}=|\A_n^{(m)}|$.  
We then recover the known case $m=2$ from this state system and work out the case $m=3$ explicitly. 
On the asymptotic side, we prove that the exponential growth constant exists for
every $m$; for $m\ge2$ it is obtained from the spectral radii of the two cyclic components with more than one vertex in the state system.
We determine the simple-pole asymptotics for $m=2$ and $m=3$, and we prove that the growth constants are nondecreasing in $m$, strictly smaller than the Catalan growth constant $4$ for every finite $m$, and converge to $4$ as $m\to\infty$.
\end{abstract}


\section{Introduction}

A permutation $\pi=\pi_1\cdots\pi_n\in S_n$ is said to contain the pattern
$132$ if there exist indices $i<j<k$ such that $\pi_i<\pi_k<\pi_j$.
If no such triple exists, then $\pi$ is said to avoid $132$.
We write
\[
        \Av_n(132)=\{\pi\in S_n:\pi\text{ avoids the pattern }132\}.
\]
Here \(\Av_0(132)\) consists of the empty permutation.
The set $\Av_n(132)$ is counted by the $n$th Catalan number
$C_n=\frac{1}{n+1}\binom{2n}{n}$.
Equivalently, the permutation class
\[
        \Av(132)=\bigcup_{n\ge0}\Av_n(132)
\]
is one of the standard Catalan classes; see Simion and Schmidt
\cite{SimionSchmidt1985} and B\'ona \cite{Bona2012}.
A standard proof of this Catalan enumeration decomposes a $132$-avoiding
permutation at its maximum element.  If $n$ occurs in position $k$, then the
entries to the left of $n$ are all larger than the entries to the right of $n$,
and the two sides are again $132$-avoiding after standardization.  Hence
\[
        |\Av_n(132)|
        =\sum_{k=1}^{n}|\Av_{k-1}(132)|\,|\Av_{n-k}(132)|,
\]
which is the usual Catalan recurrence.

Nadler \cite{Nadler2026} recently studied the enumeration of $132$-avoiding
permutations satisfying the local adjacency constraint
\[
        |\pi_{i+1}-\pi_i|\le m
\]
for all $1\le i<n$.
We call this the $m$-bounded condition. 
For fixed $m$, let
\[
        \mathcal{A}_n^{(m)}
        =
        \{\pi\in\Av_n(132):|\pi_{i+1}-\pi_i|\le m\text{ for all }1\le i<n\}.
\]
We also write
\[
        a_n^{(m)}=|\mathcal{A}_n^{(m)}|,
        \qquad
        A^{(m)}(x)=\sum_{n\ge0}a_n^{(m)}x^n.
\]
The $m$-bounded condition is a local condition on adjacent
positions, whereas pattern avoidance is a global relative-order condition.
As Nadler observed, the interaction between these two constraints has a simple
but important consequence: if the maximum element is not the final entry, then
its position is at most $m$. 
Nadler worked out the case $m=2$ and derived a rational ordinary generating
function.

Our main result is a finite recursive system for the enumeration of
\(m\)-bounded \(132\)-avoiders for every fixed \(m\),
obtained by introducing endpoint states into the standard decomposition of
\(132\)-avoiders at the maximum element.  For a permutation \(\tau\) of length
\(N\ge1\), we call \(N-\tau_1\) and \(N-\tau_N\) its first and last endpoint
deficiencies, respectively.

The reason for using two endpoint deficiencies is best seen by asking how
endpoint threshold conditions transform under the maximum decomposition.
Thus, for the moment, suppose that a permutation \(\pi\in\A_n^{(m)}\) is required
to satisfy two prescribed endpoint bounds
\[
n-\pi_1\le u,\qquad n-\pi_n\le v.
\]
Let the maximum element \(n\) occur at position \(k\), so that \(\pi_k=n\).
By the positional restriction above, \(k\in\{1,\ldots,m\}\cup\{n\}\).

First assume that \(k<n\), and write \(\pi=L\,n\,R\).  Then the right block
\(R\) has length \(n-k\) and uses the values \(1,\ldots,n-k\).  The boundary
adjacency condition between \(n\) and the first entry of \(R\) becomes
\[
n-R_1\le m
\quad\Longleftrightarrow\quad
(n-k)-R_1\le m-k.
\]
Thus the first endpoint threshold for the recursive right block is \(m-k\).
The prescribed last endpoint bound for the whole permutation transforms in the
same way:
\[
n-R_{\last}\le v
\quad\Longleftrightarrow\quad
(n-k)-R_{\last}\le v-k.
\]
Hence the recursive call on the right block must remember both a first endpoint
threshold and a last endpoint threshold.

The left block \(L\) has length \(k-1\) and uses the values \(n-k+1,\ldots,n-1\).
Since the difference between the entry of \(L\) and \(n\) is at most \(k-1 \le m-1\), the boundary adjacency condition between \(n\) and the last entry of \(L\) is automatically met.
Thus it contributes only finitely many possibilities, depending only on the first endpoint bound.

If instead the maximum element is appended at the final position, say
\(\pi=\tau n\), then the new boundary adjacency condition is
\[
        n-\tau_{\last}\le m
        \quad\Longleftrightarrow\quad
        (n-1)-\tau_{\last}\le m-1,
\]
while the first endpoint bound transforms as
\[
n-\pi_1\le u
\quad\Longleftrightarrow\quad
(n-1)-\tau_1\le u-1. 
\]
Thus the append case also transforms endpoint thresholds into endpoint
thresholds. 

These observations motivate the endpoint-restricted sets
\[
\mathcal{T}_{u,v}^{(m)}(n)
=
\left\{ \pi\in\A_n^{(m)}:n-\pi_1\le u,\ n-\pi_n\le v \right\}.
\]
We derive a recursive system for these sets,
with \(\mathcal{A}^{(m)}_n=\mathcal{T}_{\infty, \infty}^{(m)}(n)\).
See \Cref{sec:general} for the precise definitions and the full recursion.

The resulting finite system of recursive equations proves that $A^{(m)}(x)$ is
rational for every fixed $m$.  Equivalently, the sequence $a_n^{(m)}$ satisfies
an eventual constant-coefficient linear recurrence for every fixed $m$.  In particular,
this extends the finite-state phenomenon observed by Nadler for $m=2$ to all fixed values of $m$.
In addition, the same finite system determines the exponential growth of $a_n^{(m)}$, which is also proved to converge to the Catalan growth constant $4$ as $m\to\infty$.

This result should be compared with the study of Gillespie, Monks, and Monks \cite{GillespieMonksMonks2020},
where enumeration of $k$-boundedness is considered without pattern avoidance
but with anchoring condition $\pi_1 = 1$ and $\pi_n = n$.
They proved rationality of the generating functions for each fixed $k$ by deriving a recursive system. 
Our finite system is of a different kind.  It is not obtained from a general
transfer matrix for bounded adjacent differences alone, but from the standard
decomposition of $132$-avoiders at the maximum element.  The bounded-adjacency
condition is recorded through two endpoint-deficiency thresholds, which is what
allows the recursion to close on a finite state space.

This paper is organized as follows.  \Cref{sec:general} proves the general
finite-state recursion, the rationality theorem, the component structure of the
associated dependency graph, and a recurrence-order bound.  \Cref{sec:examples}
works out explicit cases: $m=1$, $m=2$, and $m=3$, using this component structure to
organize the computation.  
\Cref{sec:growth} discusses exponential growth from the cyclic components of the dependency graph. 
It first assigns two values called component radii and component growth rates to the cyclic components, then proves that the largest component growth rate among the two cyclic components with more than one vertex is the exponential growth constant for every fixed $m\ge2$.  
The same section then gives the stronger simple-pole asymptotics for $m=2,3$, proves the large-$m$ comparison estimates and reports numerical growth constants up to $m=100$.
\Cref{sec:connections} records further connections and open directions.

\section{The general finite-state recursion}\label{sec:general}

The purpose of this section is to turn the Catalan maximum decomposition into
a finite linear recursion.  The maximum element $n$ separates a
$132$-avoider into a left block of larger values and a right block of smaller
values.  The adjacency constraint strongly restricts this decomposition: unless
$n$ is the final entry, its position is at most $m$.  Thus the left block is
bounded in length, while the right block is the only part that can remain
large.  The endpoint states introduced below record exactly the boundary
information needed to recurse on this right block.

Throughout this section, fix $m\ge1$.  We use the notation
$\A_n^{(m)}$, $a_n^{(m)}$, and $A^{(m)}(x)$ introduced in the Introduction,
with the convention $a_0^{(m)}=1$.

\subsection{Maximum decomposition}

We first recall the classical maximum decomposition of $132$-avoiders.  This
decomposition is the source of the recursion: after removing the maximum
element, the right block remains a smaller $132$-avoider, while the left block
will be shown to have bounded length under the adjacency constraint.
Throughout, if $X=X_1\cdots X_t$ is a block, then $X_i$ denotes its $i$th
entry; when $t\ge1$, we write $X_{\last}=X_t$.  We use the standard notion
of standardization: the standardization of a block $X$ with distinct entries
is the permutation with the same relative order as $X$, obtained by replacing
the smallest entry by $1$, the second smallest by $2$, and so on.

The next two lemmas isolate these two ingredients: the usual maximum
decomposition and the position restriction imposed by the adjacency constraint.

\begin{lemma}[Maximum decomposition]\label{lem:maxdecomp}
Let $\pi\in\Av_n(132)$, and suppose that $n$ occurs in position $k$. 
Let $L$ and $R$ be the blocks to the left and to the right of $n$,
respectively, so that $\pi=L\,n\,R$. 
Then every entry of $L$ is larger than every entry of
$R$.  Consequently, $L$ uses the values $n-k+1,n-k+2,\ldots,n-1$, the block
$R$ uses the values $1,2,\ldots,n-k$, and the standardizations of $L$ and $R$
both avoid $132$.
\end{lemma}

\begin{proof}
If there were entries $x\in L$ and $y\in R$ with $x<y$, then $x,n,y$ would
occur in this order and would satisfy $x<y<n$, forming a $132$-pattern.  Hence
all entries of $L$ exceed all entries of $R$. 
The assertions about the value sets follow because the entries form the set
$\{1,2,\ldots,n\}$.
Any $132$-pattern inside either block would also be a $132$-pattern in $\pi$,
so both standardized blocks avoid $132$.
\end{proof}

\begin{lemma}[Position restriction]\label{lem:position}
Let $\pi\in\A_n^{(m)}$, and suppose that the maximum element $n$ occurs in
position $k$.  If the block to the right of $n$ is nonempty, then $k\le m$.
Equivalently, the maximum element occurs either in one of the positions
$1,\ldots,m$ or in the final position.
\end{lemma}

\begin{proof}
If the right block is nonempty, then by \Cref{lem:maxdecomp} the entry
immediately to the right of $n$ is at most $n-k$.  Therefore \(n-\pi_{k+1}\ge k\).
The adjacency condition gives $n-\pi_{k+1}\le m$, hence $k\le m$.  If the right
block is empty, then $n$ occurs in the final position.
\end{proof}

\subsection{Endpoint states}

We now define the auxiliary endpoint states used in the recursion.  Set
$B_m=\{0,1,\ldots,m-1,\infty\}$.  For $u,v\in\mathbb Z\cup\{\infty\}$ and
$n\ge1$, define
\begin{align}
\mathcal{T}_{u,v}^{(m)}(n)
&=
\left\{ \pi\in\A_n^{(m)}:n-\pi_1\le u,\ n-\pi_n\le v \right\},\label{eq:T-set-def}\\
T_{u,v}^{(m)}(n)
&=
\left| \mathcal{T}_{u,v}^{(m)}(n) \right|.
\end{align}
Here an inequality with threshold $\infty$ is regarded as vacuous.  These
states are cumulative and need not form a disjoint partition.  The actual
states are the pairs $(p,q)\in B_m^2$.  The extension to arbitrary integer
thresholds is only a notational convention.  In particular,
$T_{u,v}^{(m)}(n)=0$ if $u<0$ or $v<0$, since
$n-\pi_1,n-\pi_n\ge0$.
We also use the convention that $\infty-r=\infty$ for every integer $r\ge0$.
Thus, if $q\in B_m$, the expression $q-r$ is either an element of $B_m$ or a
negative integer, and in the latter case the corresponding term is zero.

The choice of these states is dictated by the maximum decomposition.  Suppose
that the maximum $n$ occurs in position $k<n$, and let $R$ be the right block.
Then $R$ has length $n-k$ and uses the values $1,\ldots,n-k$.  The boundary
adjacency condition between $n$ and the first entry of $R$ is
\[
n-R_1\le m,
\]
which is equivalent to
\[
(n-k)-R_1\le m-k.
\]
Thus the first endpoint threshold for the recursive right block is $m-k$.
Similarly, the final endpoint condition for the whole permutation,
\[
n-R_{\last}\le q,
\]
is equivalent to
\[
(n-k)-R_{\last}\le q-k.
\]
If $q-k<0$, this last endpoint condition is impossible; by the convention
above, the corresponding $T$-term is automatically zero.
Thus the last endpoint threshold is shifted by subtracting $k$.  In this sense
the endpoint states are closed under the maximum decomposition.

It remains to encode the finitely many choices for the left block.  If the
maximum $n$ occurs in position $k$, then the standardized left block has length
$k-1$, and its only dependence on the state $(p,q)$ is through the condition
$n-\pi_1\le p$.  For $1\le k\le m$ and $p\in B_m$, define
\begin{equation}
c_{k,p}=
\begin{cases}
1, & k=1,\\
\left| \left\{ \sigma\in\Av_{k-1}(132):k-\sigma_1\le p \right\} \right|, & k\ge2.
\end{cases}
\end{equation}
If $p=\infty$, the endpoint condition is vacuous.

\begin{remark}[Elementary facts about $c_{k,p}$]
\label{rem:ckp-facts}
The coefficients $c_{k,p}$ have the following elementary properties, which will
be used later.  First, $c_{1,p}=1$ for every $p\in B_m$.  Assume $k\ge2$.  Then
$c_{k,p}$ is nondecreasing in $p$, and
\[
        0\le c_{k,p}\le C_{k-1}.
\]
Moreover,
\[
        c_{k,p}>0
        \quad\Longleftrightarrow\quad
        p\ge1,
\]
where $p=\infty$ is regarded as satisfying $p\ge1$.  Indeed, $c_{k,0}=0$
because $k-\sigma_1\ge1$ for every $\sigma\in S_{k-1}$.  Conversely, if
$p\ge1$, then the decreasing permutation $(k-1)(k-2)\cdots1$ is $132$-avoiding
and satisfies $k-\sigma_1=1$, so $c_{k,p}>0$.

Finally,
\[
        c_{k,p}=C_{k-1}
\]
whenever $p\ge k-1$ or $p=\infty$.  In particular, since $k\le m$ in the
recursion,
\[
        c_{k,m-1}=c_{k,\infty}=C_{k-1}.
\]
For explicit computation, one may use the Catalan-triangle distribution of the
first entry \cite{DeSantisEtAl2013}.  If $k\ge2$ and $p$ is finite, then
\[
        c_{k,p}=
        \sum_{d=1}^{\min(p,k-1)}
        \frac{d}{k-1}\binom{2k-d-3}{k-2},
\]
with the convention that an empty sum is zero; also $c_{k,\infty}=C_{k-1}$.
\end{remark}

With this notation, maximum-position cases give the finite recursion.

\begin{theorem}[Finite-state recursion]\label{thm:recursion}
Fix $m\ge1$.  For every $n\ge2$ and every $p,q\in B_m$,
\begin{equation}
T_{p,q}^{(m)}(n)
=
\sum_{k=1}^{\min(m,n-1)} c_{k,p}
T_{m-k,q-k}^{(m)}(n-k)
+
T_{p-1,m-1}^{(m)}(n-1),
\end{equation}
where $\infty-r=\infty$ and any term with a negative threshold is interpreted
as zero.  In particular, for $n\ge m+1$ this becomes
\begin{equation}
T_{p,q}^{(m)}(n)
=
\sum_{k=1}^{m} c_{k,p}
T_{m-k,q-k}^{(m)}(n-k)
+
T_{p-1,m-1}^{(m)}(n-1).
\end{equation}
\end{theorem}

\begin{proof}
Let $\pi\in\A_n^{(m)}$ be counted by $T_{p,q}^{(m)}(n)$, and locate the maximum
entry $n$.  By \Cref{lem:position}, either $n$ occurs in one of the positions
$1,\ldots,\min(m,n-1)$ and the right block is nonempty, or $n$ occurs in the
final position.  We count these disjoint cases separately.

First suppose that $n$ occurs in position $k$ with
$1\le k\le \min(m,n-1)$, and let
$\pi=L\,n\,R$ be the corresponding maximum decomposition.
By \Cref{lem:maxdecomp}, the left block $L$ uses the values
$n-k+1,\ldots,n-1$, and the right block $R$ uses the values $1,\ldots,n-k$.
Let $\sigma\in\Av_{k-1}(132)$ be the standardization of $L$; when $k>1$,
this means that $L_i=n-k+\sigma_i$ for $1\le i\le k-1$.
Because $k\le m$, the adjacency condition inside $L$ is
automatic: any two entries of $L$ differ by at most $k-2\le m-2$.  If $L$ is
nonempty, the step from the last entry of $L$ to $n$ is also automatic, since
its size is at most $k-1\le m-1$.

The condition $n-\pi_1\le p$ depends only on the first entry of $L$.  If
$k=1$, then $\pi_1=n$, so this condition is automatic.  If $k>1$, then
\[
n-\pi_1=k-\sigma_1.
\]
Thus the number of possible left blocks compatible with $n-\pi_1\le p$ is
$c_{k,p}$.

The right block $R$ lies in $\A_{n-k}^{(m)}$.  The boundary adjacency condition
between $n$ and $R$ is equivalent to
\[
n-R_1\le m
\quad\Longleftrightarrow\quad
(n-k)-R_1\le m-k.
\]
Similarly, the condition $n-\pi_n\le q$ is equivalent to
\[
n-R_{\last}\le q
\quad\Longleftrightarrow\quad
(n-k)-R_{\last}\le q-k.
\]
Therefore, for this fixed $k$, the right block contributes
$T_{m-k,q-k}^{(m)}(n-k)$, and the factor $c_{k,p}$ counts the compatible left
blocks.

Conversely, choose a standardized left block $\sigma$ counted by $c_{k,p}$
and a right block counted by $T_{m-k,q-k}^{(m)}(n-k)$.  Relabel the left block
by setting $L_i=n-k+\sigma_i$ for $1\le i\le k-1$, place $n$ between the two
blocks, and keep the right block on the values $1,\ldots,n-k$.
We check that no new $132$-pattern is created.  A pattern using the maximum
$n$ would have to use it as the middle index, that is, as the largest entry in
positions $i<k<j$; this is impossible because every entry of $L$ is larger than
every entry of $R$.  A pattern not using $n$ is either contained entirely in one
block, where it is excluded by assumption, or uses entries from both $L$ and
$R$.  In the latter case, writing the three indices as $i<j<\ell$, the first
index that lies before $n$ is in $L$ and the third index that lies after $n$ is
in $R$, so $\pi_i>\pi_\ell$.  This is incompatible with the $132$ condition
$\pi_i<\pi_\ell<\pi_j$.  Hence no cross-block $132$-pattern is created.  The adjacency and endpoint
conditions are exactly the inequalities checked above.  Hence this product
counts precisely the permutations whose maximum is in position $k$.

It remains to consider the case where $n$ occurs in the final position.  Then
$\pi=\tau n$ with $\tau\in\A_{n-1}^{(m)}$.  The new adjacency condition is
equivalent to
\[
n-\tau_{\last}\le m
\quad\Longleftrightarrow\quad
(n-1)-\tau_{\last}\le m-1.
\]
Moreover,
\[
n-\pi_1=1+((n-1)-\tau_1),
\]
so the condition $n-\pi_1\le p$ is equivalent to
\[
(n-1)-\tau_1\le p-1.
\]
The last endpoint condition is automatic because $\pi_n=n$.  Hence the final
position case contributes $T_{p-1,m-1}^{(m)}(n-1)$.

Conversely, any $\tau$ counted by $T_{p-1,m-1}^{(m)}(n-1)$ gives a valid
permutation $\tau n$: appending the global maximum at the final position
creates no new $132$-pattern, and the new adjacency and endpoint conditions are
exactly the inequalities displayed above.

The cases are disjoint and exhaustive by \Cref{lem:position}, so summing them
proves the recursion.
\end{proof}

We now convert the finite recursion into a finite linear system of ordinary
generating functions.  After passing to generating functions, the shifts
$n\mapsto n-k$ become multiplication by powers of $x$, while the finitely many
initial values contribute only polynomials.  Hence the finitely many state
generating functions satisfy a linear system whose coefficients are polynomials
in $x$ with rational coefficients.

\begin{corollary}[Rationality and effective computability]\label{cor:rational}
For every fixed $m\ge1$, the generating function $A^{(m)}(x)$ is rational.
Moreover, it is effectively computable from a linear system of size at most
$(m+1)^2$.
\end{corollary}

\begin{proof}
For $p,q\in B_m$, define
\begin{equation}
F_{p,q}^{(m)}(x)=\sum_{n\ge1}T_{p,q}^{(m)}(n)x^n.
\end{equation}
For notational convenience, set $T_{u,v}^{(m)}(n)=0$ for all $n\ge1$ whenever
one of the finite thresholds $u,v$ is negative, and set
$F_{u,v}^{(m)}(x)=0$ in this case.  We also use the convention
$\infty-r=\infty$ for every integer $r\ge0$.

By \Cref{thm:recursion}, for every $n\ge2$ and every $p,q\in B_m$,
\[
T_{p,q}^{(m)}(n)
=
\sum_{k=1}^{\min(m,n-1)}c_{k,p}T_{m-k,q-k}^{(m)}(n-k)
+
T_{p-1,m-1}^{(m)}(n-1).
\]

The initial contribution is the same for all states.  Indeed, for $n=1$ there
is only the permutation $1$.  It avoids $132$, the adjacency condition is
vacuous, and
\[
1-\pi_1=1-\pi_n=0.
\]
Thus
\[
T_{p,q}^{(m)}(1)=1
\]
for every $p,q\in B_m$.  Hence the left-hand side of the summed recursion is
\[
\sum_{n\ge2}T_{p,q}^{(m)}(n)x^n
=
F_{p,q}^{(m)}(x)-x.
\]
On the right-hand side, for each fixed $k$, the condition
$1\le k\le\min(m,n-1)$ is equivalent to $1\le k\le m$ and $n\ge k+1$.
Therefore
\begin{align*}
\sum_{n\ge2}
\sum_{k=1}^{\min(m,n-1)}
c_{k,p}T_{m-k,q-k}^{(m)}(n-k)x^n
&=
\sum_{k=1}^{m}
c_{k,p}x^k
\sum_{r\ge1}T_{m-k,q-k}^{(m)}(r)x^r \notag\\
&=
\sum_{k=1}^{m}c_{k,p}x^kF_{m-k,q-k}^{(m)}(x).
\end{align*}
Also,
\[
\sum_{n\ge2}T_{p-1,m-1}^{(m)}(n-1)x^n
=
xF_{p-1,m-1}^{(m)}(x).
\]
Thus
\[
F_{p,q}^{(m)}(x)
-
\sum_{k=1}^{m}c_{k,p}x^kF_{m-k,q-k}^{(m)}(x)
-
xF_{p-1,m-1}^{(m)}(x)
=
x.
\]

Let $F_m(x)$ be the column vector whose entries are the $(m+1)^2$ functions
$F_{p,q}^{(m)}(x)$, indexed by $B_m^2$ in any fixed order.  In addition, let
$\mathbf{1}_d$ $(d>0)$ denote the column vector of length $d$ whose entries are
all equal to $1$.  Then the preceding equations form a finite linear system
\begin{equation}
M_m(x)F_m(x)=x\mathbf{1}_{(m+1)^2}.
\end{equation}
The coefficient matrix \(M_m(x)\) has rows and columns indexed by $B_m^2$, and its entries
are
\begin{align}
M_{(p,q),(r,s)}(x)
={}&
\chi_{(p,q)=(r,s)}
-
\sum_{k=1}^{m}c_{k,p}x^k\chi_{(r,s)=(m-k,q-k)}
-
x\chi_{(r,s)=(p-1,m-1)},
\end{align}
where $\chi_E$ is equal to $1$ if the condition $E$ holds and $0$ otherwise.
If one of the pairs $(m-k,q-k)$ or $(p-1,m-1)$ lies outside $B_m^2$ because of
a negative finite threshold, the corresponding indicator is understood to be
zero.

The constant term of $M_m(x)$ is the identity matrix.  Hence
$\det M_m(x)$ has constant term $1$, and in particular it is not the zero
polynomial.  Therefore $M_m(x)$ is invertible over the field of rational
functions in $x$.  By Cramer's rule, every component of
\begin{equation}
F_m(x)=M_m(x)^{-1}x\mathbf{1}_{(m+1)^2}
\end{equation}
is a rational function of $x$.

Finally, the state $(\infty,\infty)$ imposes no endpoint restrictions, so
\begin{equation}
T_{\infty,\infty}^{(m)}(n)=a_n^{(m)}
\end{equation}
for all $n\ge1$.  Therefore
\begin{equation}
A^{(m)}(x)=1+F_{\infty,\infty}^{(m)}(x).
\end{equation}
Thus $A^{(m)}(x)$ is rational.  The construction is effective, since all
entries of $M_m(x)$ are explicitly determined by the finitely many coefficients
$c_{k,p}$, and the right-hand side vector $x\mathbf{1}_{(m+1)^2}$ is explicit.
\end{proof}

\begin{remark}[Relation to Nadler's finite-state formulation]
Nadler's conjecture is phrased in terms of finitely many configurations near the
left end of the permutation.  The state space above is different: it uses both
endpoints and consists of cumulative threshold states.  Thus \Cref{cor:rational} proves the enumerative finite-state consequence,
namely rationality of the generating function, by a two-sided
endpoint-state method.  No minimality of this state space is asserted.
\end{remark}

\subsection{Dependency graph and cyclic components}\label{subsec:dependency-graph}

The rationality statement above only uses the fact that the endpoint-state recursion is
finite.  For later arguments, however, we also need the internal structure
of this finite system.  In particular, we need to know which state equations can
feed back into one another after repeated substitution.  This dependence is
encoded by a directed graph on the endpoint states.  Its strongly connected
components give a block-triangular ordering of the linear system, and the
components containing directed cycles are precisely the feedback components.  These components
organize the explicit computations in \Cref{sec:examples}, give the
recurrence-order bound below,
and provide the input for the component-growth analysis in \Cref{sec:growth}, which
is used to estimate the exponential growth of $a_n^{(m)}$.

To make this dependence explicit, set
\begin{equation}\label{eq:W-def}
        M_m(x)=I-W_m(x).
\end{equation}
Thus $W_m(x)$ is the polynomial matrix whose entries are given by
\begin{align}\label{eq:W-entries}
W_{(p,q),(r,s)}^{(m)}(x)
={}&
\sum_{k=1}^{m}c_{k,p}x^k\chi_{(r,s)=(m-k,q-k)}
+
x\chi_{(r,s)=(p-1,m-1)}.
\end{align}
Equivalently, the state system can be written as
\begin{equation}\label{eq:F-W-system}
        F_m(x)=x\mathbf{1}_{(m+1)^2}+W_m(x)F_m(x).
\end{equation}
The powers of $x$ record the size lost in the recursive call: the term with
$x^k$ comes from a recursive call on an object of size $n-k$.

We define the dependency graph $\Gamma_m$ as follows. 
Its vertex set is $B_m^2$.
There is a directed edge $(r,s)\rightarrow (p,q)$
if $W_{(p,q),(r,s)}^{(m)}(x)$ is not the zero polynomial.  Thus the direction
is from a state appearing on the right-hand side of the equation to the state
being computed on the left-hand side.  A set of vertices is called strongly
connected if, for any two vertices in the set, there are directed paths in both
directions between them.  A strongly connected component is a maximal strongly
connected set.  We call a strongly connected component cyclic if it contains a
directed cycle.  Equivalently, a cyclic component either has more than one
vertex or is a singleton with a self-loop.  A singleton vertex without a
self-loop is also a strongly connected component in the graph-theoretic sense,
but it is not cyclic.  In the linear system, the cyclic components are exactly
the feedback components: these are the components in which a state can depend on itself
after finitely many substitutions.

For later use, if \(S\subseteq B_m^2\), let \(W_S(x)\) denote the principal submatrix of \(W_m(x)\) whose rows and columns are indexed by the states in \(S\).
When \(S=C\) is a strongly connected component of \(\Gamma_m\), we call \(W_C(x)\) the component matrix of \(C\). 

The recursion gives two kinds of edges in $\Gamma_m$,
and it will be useful to distinguish them.
A $k$-split edge, with $1\le k\le m$, is an edge produced by the term
$c_{k,p}x^kF^{(m)}_{m-k,q-k}(x)$
in the equation for $F^{(m)}_{p,q}(x)$.  Hence its source is
$(m-k,q-k)$ and its target is $(p,q)$, provided that the source lies in
$B_m^2$ and $c_{k,p}>0$.  This edge corresponds to the case in
\Cref{thm:recursion} in which the maximum element has a nonempty right block
and occurs in position $k$.  An append-edge is an edge produced by the term
$xF^{(m)}_{p-1,m-1}(x)$
in the equation for $F^{(m)}_{p,q}(x)$.  Its source is $(p-1,m-1)$ and its
target is $(p,q)$, provided that the source lies in $B_m^2$. 
This edge corresponds to the case in which
the maximum element is appended in the final position.

\begin{figure}[H]
\centering
\includegraphics[width=0.72\textwidth]{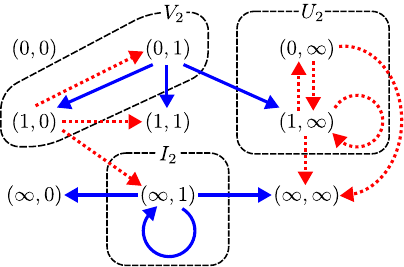}
\caption{Illustration of the dependency graph $\Gamma_2$.  The dotted red arcs
are split-edges; the two arcs starting at $(0,\infty)$ are $2$-split edges, and
the remaining dotted arcs are $1$-split edges.  The solid blue arcs are
append-edges.  The dashed enclosures indicate the cyclic strongly
connected components $U_2$, $V_2$, and $I_2$, identified in \Cref{prop:recurrent-blocks}.}
\label{fig:dependency2}
\end{figure}
\Cref{fig:dependency2} shows an example of $\Gamma_2$.
From this illustration, we can observe that
not every state generating function $F_{p,q}^{(2)}(x)$ is needed to determine
$F_{\infty,\infty}^{(2)}(x)$.
With the edge orientation above, only the states $(0, 1)$, $(1, 0)$,
$(0, \infty)$, $(1, \infty)$, and $(\infty, 1)$ can reach $(\infty, \infty)$.
These states are grouped into $U_2$, $V_2$, and $I_2$,
each of which is a cyclic strongly connected component.
The proposition below gives the corresponding classification for all $m\ge2$.

\begin{proposition}[Cyclic components of the dependency graph]\label{prop:recurrent-blocks}
Assume $m\ge2$.  In the dependency graph $\Gamma_m$, the cyclic strongly connected
components are the following:
\begin{align}
U_m&=\{(p,\infty):0\le p\le m-1\},\label{eq:U-def}\\
V_m&=\{(p,q):0\le p,q\le m-1,\ q<p\}
      \cup
      \{(p,m-1):0\le p\le m-2\},\label{eq:V-def}\\
I_m&=\{(\infty,m-1)\}.\label{eq:I-def}
\end{align}
The sets $U_m$ and $V_m$ are strongly connected components.  The component
$I_m$ is a single vertex with a self-loop.  Every vertex in
$B_m^2\setminus(U_m\cup V_m\cup I_m)$ is a singleton strongly connected
component without a self-loop.  Consequently, after ordering the states by
strongly connected components, the system \eqref{eq:F-W-system} is block
triangular, and the only cyclic components with more than one vertex are $V_m$ and $U_m$.
\end{proposition}

\begin{proof}
We divide the vertices according to whether their coordinates are finite or
infinite.  Call a vertex $(r,s)$ finite--finite if both coordinates are finite,
that is, if $0\le r,s\le m-1$.

First consider a finite--finite vertex $(r,s)$.  We determine exactly when it
can be the source of an edge.  Suppose that $(r,s)$ is the source of a
$k$-split edge.  Then, for some target state $(p,q)$, the source identity is
\[
        (r,s)=(m-k,q-k).
\]
Thus $k=m-r$ and $q=s+k=s+m-r$.  Since $q$ must lie in
$\{0,1,\ldots,m-1\}$, this is possible only if
\[
        s+m-r\le m-1,
\]
or equivalently $s<r$.  Conversely, if $s<r$, then $k=m-r$ satisfies
$1\le k\le m$, and $q=s+k$ lies in $\{0,1,\ldots,m-1\}$.  Moreover there is
some target first threshold $p$ with $c_{k,p}>0$: if $k=1$, then $c_{1,p}=1$
for every $p\in B_m$; if $k\ge2$, then choosing $p=1$ works, because the
132-avoiding permutation of length $k-1$ that starts with its maximum has
$k-\sigma_1=1$.  Hence a finite--finite vertex is the source of a $k$-split edge
for some $k$ if and only if $s<r$.

Next suppose that the finite--finite vertex $(r,s)$ is the source of an
append-edge.  Then, for some target state $(p,q)$, the source identity is
\[
        (r,s)=(p-1,m-1).
\]
Since $p$ is finite in this case, this is possible if and only if
\[
        s=m-1,
        \qquad
        0\le r\le m-2.
\]
Conversely, every such $(r,m-1)$ is the source of append-edges.

It follows that the finite--finite vertices which can be sources of edges are
precisely the vertices in
\[
V_m=
\{(r,s):0\le r,s\le m-1,\ s<r\}
\cup
\{(r,m-1):0\le r\le m-2\}.
\]
The remaining finite--finite vertices are
\[
        \{(r,s):0\le r\le s\le m-2\}\cup\{(m-1,m-1)\}.
\]
They have no outgoing edges at all, and hence each of them is a singleton
strongly connected component without a self-loop.

We now prove that $V_m$ is strongly connected.  Put $b=(0,m-1)$.  We first
show that $b$ reaches every vertex of $V_m$.  If the target to reach is $b$ itself, there is
nothing to prove.  Otherwise let the target be $(r,s)\in V_m$ with $r\ge1$.
Starting at $b$, use append-edges
\[
(0,m-1)\to(1,m-1)\to\cdots\to(r-1,m-1),
\]
omitting this displayed path when $r=1$, and then use the append-edge
\[
        (r-1,m-1)\to(r,s).
\]
This is allowed because $r-1\le m-2$.
In \Cref{fig:V_m}, these paths of append-edges are illustrated by solid blue arcs.
\begin{figure}[H]
\centering
\includegraphics[width=1.0\textwidth]{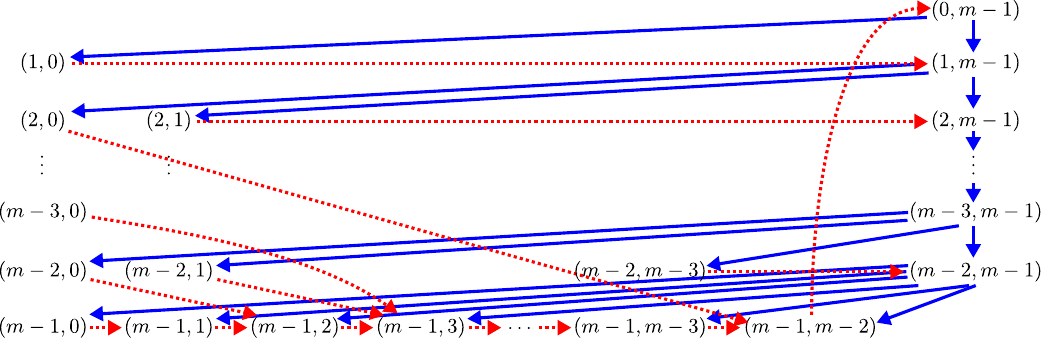}
\caption{The finite--finite feedback component $V_m$.  The vertices shown are the
states $(r,s)$ with $0\le s<r\le m-1$ and the states $(r,m-1)$ with
$0\le r\le m-2$.  Only the arrows
used in the proof of strong connectivity are drawn: solid blue arrows are
append-edges, and dotted red arrows are $k$-split.}
\label{fig:V_m}
\end{figure}

Conversely, every vertex of $V_m$ reaches $b$.  First consider a vertex of the
form $(r,m-1)$ with $0\le r\le m-2$.  If $r<m-2$, use append-edges to reach
$(m-2,m-1)$.  Then use the append-edge
\[
        (m-2,m-1)\to(m-1,m-2),
\]
and finally use the $1$-split edge with target first threshold $0$ to obtain
\[
        (m-1,m-2)\to(0,m-1)=b.
\]
This also covers $m=2$, where $(m-2,m-1)=b$.

It remains to consider a vertex $(r,s)$ with $s<r$.  If $s\le r-2$, then the
$(m-r)$-split edge with target first threshold $m-1$ gives
\[
        (r,s)\to(m-1,s+m-r),
\]
where $s+m-r\le m-2$.  Repeatedly using $1$-split edges with target first threshold
$m-1$ increases the second coordinate until we reach $(m-1,m-2)$; the preceding
edge $(m-1,m-2)\to b$ then finishes the path.  If instead $s=r-1$, there are
two cases.  When $r=m-1$, the vertex is already $(m-1,m-2)$, so the edge just
described sends it directly to $b$.  When $r\le m-2$, the $(m-r)$-split edge with
target first threshold $1$ gives
\[
        (r,r-1)\to(1,m-1),
\]
and we have already shown that vertices of the form $(t,m-1)$ reach $b$. 
In \Cref{fig:V_m}, these paths of $k$-split edges are illustrated by dotted red arcs.
Thus every vertex of $V_m$ reaches $b$, and $V_m$ is strongly connected.

Next consider vertices whose second coordinate is $\infty$ and whose first
coordinate is finite.  These are exactly the vertices of
\[
        U_m=\{(r,\infty):0\le r\le m-1\}.
\]
From $(m-1,\infty)$, the $1$-split edges reach every $(p,\infty)$ with
$0\le p\le m-1$.  Conversely, for any $0\le r\le m-1$, the $(m-r)$-split edge with
target first threshold $m-1$ sends
\[
        (r,\infty)\to(m-1,\infty).
\]
For $r=m-1$ this is a self-loop.  Hence $U_m$ is strongly connected.

Finally consider vertices whose first coordinate is $\infty$.  Such a vertex
cannot be the source of a $k$-split edge, because every $k$-split edge has source first
coordinate $m-k$, which is finite.  It can be the source of an append-edge only
when it is $(\infty,m-1)$, corresponding to the choice $p=\infty$.  Thus
$I_m=\{(\infty,m-1)\}$ is a singleton component with a self-loop.  The same
vertex also has append-edges to $(\infty,q)$ for all $q\in B_m$, but no vertex
$(\infty,q)$ with $q\ne m-1$ has any outgoing edge.  Therefore all remaining
first-$\infty$ vertices are singleton components without self-loops.

Combining the finite--finite case, the finite--$\infty$ case, and the
first-$\infty$ case gives the stated classification of cyclic strongly
connected components.  The component decomposition of a finite directed graph
always gives a block triangular ordering of its adjacency matrix, and the same
ordering gives a block triangular form for the polynomial matrix $W_m(x)$.  The
last assertion follows.
\end{proof}

The classification in \Cref{prop:recurrent-blocks} has the following immediate
quantitative consequence.  
Acyclic singleton components contribute only \(1\times1\) identity blocks $(1)$ to $I-W_m(x)$, so only the cyclic components can contribute to the determinant of the finite linear system.  
A crude degree count therefore gives an explicit recurrence-order bound.

\begin{corollary}[A recurrence-order bound]
\label{cor:recurrence-order}
\normalfont
For $m=1$, the sequence $a_n^{(1)}$ satisfies an eventual homogeneous linear recurrence with constant coefficients of order at most \(1\).
For every $m\ge2$, the sequence $a_n^{(m)}$ satisfies an eventual homogeneous linear recurrence with constant coefficients of order at most 
\begin{equation}
        d_m
        =m|U_m|+m|V_m|+1
        =m^2+\frac{m(m-1)(m+2)}{2}+1.
\end{equation}
The number \(d_m\) is only an upper bound on the minimal recurrence order, since cancellations may occur in the rational generating function.
In particular, $d_m=O(m^3)$.
\end{corollary}

\begin{proof}
For $m=1$, every nonempty $1$-bounded permutation is monotone, so
$a_n^{(1)}=2$ for all $n\ge2$ and the claimed order-one recurrence follows.
Assume $m\ge2$.
After ordering the states by strongly connected components, the matrix
$M_m(x)=I-W_m(x)$ is block triangular.  A singleton component without a self-loop
contributes the diagonal block $(1)$ and hence contributes no factor to the
determinant.  By \Cref{prop:recurrent-blocks}, the cyclic components are $U_m$,
$V_m$, and $I_m=\{(\infty,m-1)\}$.  
Therefore, using the component-matrix notation introduced above,
\[
\det M_m(x)
=
\det(I-W_{U_m}(x))\det(I-W_{V_m}(x))\det(I-W_{I_m}(x)).
\]
Every entry of $W_m(x)$ has degree at most $m$.  Thus
\[
\deg\det(I-W_{U_m}(x))\le m|U_m|,
\qquad
\deg\det(I-W_{V_m}(x))\le m|V_m|.
\]
The component $I_m$ is the one-state component with matrix $(x)$, so
$\det(I-W_{I_m}(x))=1-x$ and has degree $1$.  Hence
\[
\deg\det M_m(x)\le m|U_m|+m|V_m|+1.
\]
The denominator of $A^{(m)}(x)=1+F_{\infty,\infty}^{(m)}(x)$ after cancellation
divides $\det M_m(x)$.  If a rational generating function has denominator of
degree at most $d$ and nonzero constant term, then its coefficients satisfy an
eventual constant-coefficient linear recurrence of order at most $d$.  This
proves the stated bound.
\end{proof}

\section{Explicit cases}\label{sec:examples}
This section illustrates how the general endpoint-state system is used in
concrete cases.  The case \(m=1\) serves as a degenerate boundary case and
checks the conventions.  The case \(m=2\) recovers Nadler's rational generating
function from the present state system.  The case \(m=3\) is the first genuinely
branching case and provides an explicit rational formula used later in the
asymptotic discussion.
\subsection{The case \texorpdfstring{$m=1$}{m=1}}

When $m=1$, a permutation satisfying $|\pi_{i+1}-\pi_i|\le1$ must move by
steps of size exactly $1$ whenever $n\ge2$.  Hence, for $n\ge2$, it is either
$12\cdots n$ or $n\cdots21$; for $n=1$ these two descriptions coincide.
Both monotone permutations avoid $132$. 
Since the empty permutation gives
$a_0^{(1)}=1$, we have
\begin{equation}
a_n^{(1)}=
\begin{cases}
1, & n=0,\\
1, & n=1,\\
2, & n\ge2.
\end{cases}
\end{equation}
Therefore
\begin{equation}
A^{(1)}(x)=1+x+2x^2+2x^3+\cdots=\frac{1+x^2}{1-x}.
\end{equation}

The same answer is obtained from the finite-state system of
\Cref{cor:rational}.  Here $B_1=\{0,\infty\}$, and the recursion in
\Cref{thm:recursion} is
\begin{equation}
\label{eq:recursion1}
T_{p,q}^{(1)}(n)=c_{1,p}T_{0,q-1}^{(1)}(n-1)+T_{p-1,0}^{(1)}(n-1) 
=T_{0,q-1}^{(1)}(n-1)+T_{p-1,0}^{(1)}(n-1)
\end{equation}
for $n\ge2$ and $p,q\in B_1$.
For $p,q\in B_1$, we have $T_{p,q}^{(1)}(1)=1$.  More generally, under the
extended-threshold convention, $T_{u,v}^{(1)}(1)=0$ if $u<0$ or $v<0$.
Order the state generating functions as
\[
F_1(x)=
\begin{pmatrix}
F_{0,0}^{(1)}(x)\\
F_{0,\infty}^{(1)}(x)\\
F_{\infty,0}^{(1)}(x)\\
F_{\infty,\infty}^{(1)}(x)
\end{pmatrix}.
\]
From \eqref{eq:recursion1}, we obtain
\[
F_{0,0}^{(1)}(x)
= T_{0,0}^{(1)}(1)x + \sum_{n\ge2}T_{0,0}^{(1)}(n)x^n
= x,
\]

\begin{align*}
F_{0,\infty}^{(1)}(x)
&= T_{0,\infty}^{(1)}(1) x + \sum_{n\ge2}T_{0,\infty}^{(1)}(n)x^n
= x + \sum_{n\ge2}T_{0,\infty}^{(1)}(n-1)x^n \notag \\
&= x + x \sum_{n\ge1}T_{0,\infty}^{(1)}(n)x^n
= x + x F_{0,\infty}^{(1)}(x),
\end{align*}
\begin{align*}
F_{\infty, 0}^{(1)}(x)
&= T_{\infty,0}^{(1)}(1) x + \sum_{n\ge2}T_{\infty, 0}^{(1)}(n)x^n
= x + \sum_{n\ge2}T_{\infty, 0}^{(1)}(n-1)x^n \notag \\
&= x + x \sum_{n\ge1}T_{\infty, 0}^{(1)}(n)x^n
= x + x F_{\infty, 0}^{(1)}(x),
\end{align*}
and
\begin{align*}
F_{\infty, \infty}^{(1)}(x)
&= \sum_{n\ge1}T_{\infty, \infty}^{(1)}(n)x^n
= T_{\infty, \infty}^{(1)}(1) x + \sum_{n\ge2}\left(T_{0, \infty}^{(1)}(n-1) + T_{\infty, 0}^{(1)}(n-1) \right) x^n \notag \\
&= x + x \sum_{n\ge1}T_{0, \infty}^{(1)}(n) x^n + x \sum_{n\ge1}T_{\infty, 0}^{(1)}(n) x^n
= x + x F_{0, \infty}^{(1)}(x) + x F_{\infty, 0}^{(1)}(x).
\end{align*}
Therefore, the linear system for $m=1$ is
\[
M_1(x)F_1(x)=x \mathbf{1}_4,
\]
where
\[
M_1(x)=
\begin{pmatrix}
1&0&0&0\\
0&1-x&0&0\\
0&0&1-x&0\\
0&-x&-x&1
\end{pmatrix}.
\]
Solving gives
\[
F_{\infty,\infty}^{(1)}(x)=\frac{x+x^2}{1-x}.
\]
Since $A^{(1)}(x)=1+F_{\infty,\infty}^{(1)}(x)$, this again gives
\begin{equation}\label{eq:m1-genfun}
A^{(1)}(x)=\frac{1+x^2}{1-x}.
\end{equation}

\subsection{The case \texorpdfstring{$m=2$}{m=2}}

We next recover the known $m=2$ case \cite{Nadler2026} from the finite-state
system.  Here
\[
B_2=\{0,1,\infty\}.
\]
The left-block coefficients are
\begin{align*}
c_{1,p}&=1,\\
c_{2,p}&=\chi_{p\ge1},
\end{align*}
where $\chi_{p\ge1}=1$ for $p=1$ and $p=\infty$, and
$\chi_{p\ge1}=0$ for $p=0$.  Thus the large-$n$ form of
\Cref{thm:recursion}, valid for $n\ge3$, is
\begin{equation}\label{eq:m2rec}
T_{p,q}^{(2)}(n)
=
T_{1,q-1}^{(2)}(n-1)
+
\chi_{p\ge1}T_{0,q-2}^{(2)}(n-2)
+
T_{p-1,1}^{(2)}(n-1)
\end{equation}
for $p,q\in B_2$.  As before, $\infty-r=\infty$, and any term with a
negative finite threshold is interpreted as zero.

For $p,q\in B_2$, define
\[
F_{p,q}^{(2)}(x)=\sum_{n\ge1}T_{p,q}^{(2)}(n)x^n.
\]
We also set $F_{u,v}^{(2)}(x)=0$ whenever one of the finite thresholds
$u,v$ is negative.  Applying the generating-function identity of
\Cref{cor:rational} gives
\begin{equation}\label{eq:m2-state-eq}
F_{p,q}^{(2)}(x)
-
xF_{1,q-1}^{(2)}(x)
-
\chi_{p\ge1}x^2F_{0,q-2}^{(2)}(x)
-
xF_{p-1,1}^{(2)}(x)
=
x.
\end{equation}

By \Cref{prop:recurrent-blocks}, the cyclic components with more than one vertex are
\[
V_2=\{(0,1),(1,0)\},
\qquad
U_2=\{(0,\infty),(1,\infty)\}.
\]
There is also the singleton self-loop component $I_2=\{(\infty,1)\}$.  We solve
these blocks in the block-triangular order $V_2$, then $I_2$, then $U_2$, and
finally the unrestricted state $(\infty,\infty)$.

The $V_2$-block is
\begin{align*}
F_{0,1}^{(2)}(x)&=x+xF_{1,0}^{(2)}(x),\\
F_{1,0}^{(2)}(x)&=x+xF_{0,1}^{(2)}(x).
\end{align*}
Solving gives
\begin{equation}\label{eq:m2-V-sol}
        F_{0,1}^{(2)}(x)=F_{1,0}^{(2)}(x)=\frac{x}{1-x}.
\end{equation}
The singleton state $(\infty,1)$ then satisfies
\[
(1-x)F_{\infty,1}^{(2)}(x)=x+xF_{1,0}^{(2)}(x),
\]
and hence
\begin{equation}\label{eq:m2-I-sol}
        F_{\infty,1}^{(2)}(x)=\frac{x}{(1-x)^2}.
\end{equation}

Next we solve the $U_2$-block.  The two equations are
\begin{align*}
F_{0,\infty}^{(2)}(x)&=x+xF_{1,\infty}^{(2)}(x),\\
F_{1,\infty}^{(2)}(x)&=x+xF_{1,\infty}^{(2)}(x)
+x^2F_{0,\infty}^{(2)}(x)+xF_{0,1}^{(2)}(x).
\end{align*}
Substituting \eqref{eq:m2-V-sol} into this system gives
\begin{align}
F_{1,\infty}^{(2)}(x)
&=
\frac{x+x^3-x^4}{(1-x)(1-x-x^3)},\label{eq:m2-U-sol1}\\
F_{0,\infty}^{(2)}(x)
&=
\frac{x(1-x+x^2)}{(1-x)(1-x-x^3)}.\label{eq:m2-U-sol0}
\end{align}
Finally, the row for $(\infty,\infty)$ is
\[
F_{\infty,\infty}^{(2)}(x)
=
x+xF_{1,\infty}^{(2)}(x)+x^2F_{0,\infty}^{(2)}(x)+xF_{\infty,1}^{(2)}(x).
\]
Using \eqref{eq:m2-I-sol}, \eqref{eq:m2-U-sol1}, and \eqref{eq:m2-U-sol0}, we obtain
\begin{equation}
F_{\infty,\infty}^{(2)}(x)
=
\frac{x-x^2+2x^3-3x^4+x^5-x^6}
{(1-x)^2(1-x-x^3)}.
\end{equation}
Since the state $(\infty,\infty)$ imposes no endpoint restrictions, we have
$A^{(2)}(x)=1+F_{\infty,\infty}^{(2)}(x)$.  Therefore
\begin{equation}
\label{eq:m2gf}
A^{(2)}(x)
=
\frac{1-2x+2x^2-x^4-x^6}{(1-x)^2(1-x-x^3)}.
\end{equation}
The first values are
$1,1,2,5,8,12,18,26,37,53,76,109,\ldots$.

\subsection{The case \texorpdfstring{$m=3$}{m=3}}\label{subsec:m3}

We now specialize the finite-state recursion to $m=3$.  We follow the same
generating-function procedure as in the cases $m=1$ and $m=2$, but we use the
component structure of \Cref{prop:recurrent-blocks} instead of displaying the full
$16\times16$ matrix.  
Here
\[
B_3=\{0,1,2,\infty\}.
\]
For $m=3$, the cyclic components with more than one vertex are
\[
V_3=\{(0,2),(1,0),(1,2),(2,0),(2,1)\},
\qquad
U_3=\{(0,\infty),(1,\infty),(2,\infty)\}.
\]
There is also the singleton self-loop component $I_3=\{(\infty,2)\}$.  We solve
first the $V_3$-block, then $I_3$, then the $U_3$-block, and finally the
unrestricted state $(\infty,\infty)$.

The left-block coefficients are
\begin{align*}
c_{1,p}&=1,\\
c_{2,p}&=\chi_{p\ge1},\\
c_{3,p}&=c_p,
\end{align*}
where
\[
c_p=
\begin{cases}
0, & p=0,\\
1, & p=1,\\
2, & p=2\text{ or }p=\infty.
\end{cases}
\]

Thus the recursion of \Cref{thm:recursion}, specialized to $m=3$, is
\begin{equation}\label{eq:m3rec-full}
T_{p,q}^{(3)}(n)
=
\sum_{k=1}^{\min(3,n-1)}c_{k,p}
T_{3-k,q-k}^{(3)}(n-k)
+
T_{p-1,2}^{(3)}(n-1)
\end{equation}
for $n\ge2$ and $p,q\in B_3$.  Equivalently, for $n\ge4$ this becomes
\[
T_{p,q}^{(3)}(n)
=
T_{2,q-1}^{(3)}(n-1)
+\chi_{p\ge1}T_{1,q-2}^{(3)}(n-2)
+c_pT_{0,q-3}^{(3)}(n-3)
+T_{p-1,2}^{(3)}(n-1).
\]
As in Section~2, $\infty-r=\infty$, and any term with a negative finite threshold is
interpreted as zero.

For $p,q\in B_3$, define
\[
F_{p,q}^{(3)}(x)=\sum_{n\ge1}T_{p,q}^{(3)}(n)x^n.
\]
We also set $F_{u,v}^{(3)}(x)=0$ whenever one of the finite thresholds
$u,v$ is negative.

We now derive the state equation from \eqref{eq:m3rec-full}.  Since the unique
permutation of length $1$ is counted by every state, we have
\[
T_{p,q}^{(3)}(1)=1
\]
for every $p,q\in B_3$.  Hence
\[
\sum_{n\ge2}T_{p,q}^{(3)}(n)x^n
=
F_{p,q}^{(3)}(x)-x.
\]
The $k=1$ term in \eqref{eq:m3rec-full} gives
\[
\sum_{n\ge2}T_{2,q-1}^{(3)}(n-1)x^n
=
x\sum_{r\ge1}T_{2,q-1}^{(3)}(r)x^r 
=
xF_{2,q-1}^{(3)}(x).
\]
The $k=2$ term appears exactly for $n\ge3$, and therefore gives
\[
\sum_{n\ge3}\chi_{p\ge1}T_{1,q-2}^{(3)}(n-2)x^n
=
\chi_{p\ge1}x^2\sum_{r\ge1}T_{1,q-2}^{(3)}(r)x^r 
=
\chi_{p\ge1}x^2F_{1,q-2}^{(3)}(x).
\]
Similarly, the $k=3$ term appears exactly for $n\ge4$, and gives
\[
\sum_{n\ge4}c_pT_{0,q-3}^{(3)}(n-3)x^n
=
c_px^3\sum_{r\ge1}T_{0,q-3}^{(3)}(r)x^r 
=
c_px^3F_{0,q-3}^{(3)}(x).
\]
Finally, the append term gives
\[
\sum_{n\ge2}T_{p-1,2}^{(3)}(n-1)x^n
=
x\sum_{r\ge1}T_{p-1,2}^{(3)}(r)x^r 
=
xF_{p-1,2}^{(3)}(x).
\]
Combining these four contributions, we obtain
\begin{equation}\label{eq:m3-state-equation}
F_{p,q}^{(3)}(x)
-
xF_{2,q-1}^{(3)}(x)
-
\chi_{p\ge1}x^2F_{1,q-2}^{(3)}(x)
-
c_px^3F_{0,q-3}^{(3)}(x)
-
xF_{p-1,2}^{(3)}(x)
=
x.
\end{equation}
This is the $m=3$ finite-state linear system.  It has $16$ equations.  The block
order stated above,
\[
        V_3\longrightarrow (\infty,2)\longrightarrow U_3\longrightarrow (\infty,\infty),
\]
explains why only the rows displayed below are needed to determine
$F_{\infty,\infty}^{(3)}(x)$.

We first solve the $V_3$-block.  From
\eqref{eq:m3-state-equation} with $(p,q)=(1,0),(2,0),(2,1)$, we obtain
\[
F_{1,0}^{(3)}(x)=x\left(F_{0,2}^{(3)}(x)+1\right),
\]
\[
F_{2,0}^{(3)}(x)=x\left(F_{1,2}^{(3)}(x)+1\right),
\]
\[
F_{2,1}^{(3)}(x)=x+xF_{2,0}^{(3)}(x)+xF_{1,2}^{(3)}(x) 
=x(1+x)\left(F_{1,2}^{(3)}(x)+1\right).
\]
The rows $(0,2)$ and $(1,2)$ of \eqref{eq:m3-state-equation} give
\[
F_{0,2}^{(3)}(x)
=
x+xF_{2,1}^{(3)}(x) 
=
x+x^2(1+x)\left(F_{1,2}^{(3)}(x)+1\right),
\]
and
\begin{align*}
F_{1,2}^{(3)}(x)
&=
x+xF_{2,1}^{(3)}(x)+x^2F_{1,0}^{(3)}(x)+xF_{0,2}^{(3)}(x) \notag\\
&=
x+x^2(1+x)\left(F_{1,2}^{(3)}(x)+1\right)
+x^3\left(F_{0,2}^{(3)}(x)+1\right)
+xF_{0,2}^{(3)}(x).
\end{align*}
Equivalently,
\begin{equation}\label{eq:m3-finite-system}
\begin{pmatrix}
1 & -x^2(1+x)\\
-x(1+x^2) & 1-x^2(1+x)
\end{pmatrix}
\begin{pmatrix}
F_{0,2}^{(3)}(x)\\
F_{1,2}^{(3)}(x)
\end{pmatrix}
=
\begin{pmatrix}
x+x^2+x^3\\
x+x^2+2x^3
\end{pmatrix}.
\end{equation}
The determinant of this reduced $V_3$-block system is
\[
\Delta_V(x)=1-x^2-2x^3-x^4-x^5-x^6.
\]
Solving \eqref{eq:m3-finite-system} gives
\begin{align*}
F_{0,2}^{(3)}(x)
&=
\frac{x(1+x+x^2+x^4+x^5)}
{1-x^2-2x^3-x^4-x^5-x^6},\\
F_{1,2}^{(3)}(x)
&=
\frac{x(1+x)(1+x+2x^2+x^4)}
{1-x^2-2x^3-x^4-x^5-x^6}.
\end{align*}

The singleton component $I_3=\{(\infty,2)\}$ is needed for the unrestricted state.  Its row in \eqref{eq:m3-state-equation} gives
\[
(1-x)F_{\infty,2}^{(3)}(x)
=
x+x^2(1+x)\left(F_{1,2}^{(3)}(x)+1\right)
+x^3\left(F_{0,2}^{(3)}(x)+1\right).
\]
Substituting the two expressions above, we obtain
\[
F_{\infty,2}^{(3)}(x)
=
\frac{x+x^2+2x^3+x^4+x^5-x^7}
{(1-x)(1-x^2-2x^3-x^4-x^5-x^6)}.
\]

We next solve the $U_3$-block, the states with second threshold $\infty$.  The row
$(0,\infty)$ of \eqref{eq:m3-state-equation} gives
\begin{equation}\label{eq:m3-f0inf}
F_{0,\infty}^{(3)}(x)=x\left(F_{2,\infty}^{(3)}(x)+1\right).
\end{equation}
The rows $(1,\infty)$ and $(2,\infty)$ give
\begin{align*}
F_{1,\infty}^{(3)}(x)
&=
x+xF_{2,\infty}^{(3)}(x)
+x^2F_{1,\infty}^{(3)}(x)
+x^3F_{0,\infty}^{(3)}(x)
+xF_{0,2}^{(3)}(x),\\
F_{2,\infty}^{(3)}(x)
&=
x+xF_{2,\infty}^{(3)}(x)
+x^2F_{1,\infty}^{(3)}(x)
+2x^3F_{0,\infty}^{(3)}(x)
+xF_{1,2}^{(3)}(x).
\end{align*}
Eliminating $F_{0,\infty}^{(3)}(x)$ by \eqref{eq:m3-f0inf}, we get the
$2\times2$ system
\begin{equation}\label{eq:m3-infinite-system}
\begin{pmatrix}
1-x^2 & -(x+x^4)\\
-x^2 & 1-x-2x^4
\end{pmatrix}
\begin{pmatrix}
F_{1,\infty}^{(3)}(x)\\
F_{2,\infty}^{(3)}(x)
\end{pmatrix}
=
\begin{pmatrix}
x\left(F_{0,2}^{(3)}(x)+1\right)+x^4\\
x\left(F_{1,2}^{(3)}(x)+1\right)+2x^4
\end{pmatrix}.
\end{equation}
The determinant of the reduced $U_3$-block coefficient matrix in \eqref{eq:m3-infinite-system} is
\[
(1-x^2)(1-x-2x^4)-x^2(x+x^4)
=
(1+x)(x^5-x^4-x^3+x^2-2x+1).
\]

Finally, the unrestricted state itself is obtained from the row
$(\infty,\infty)$:
\begin{align*}
F_{\infty,\infty}^{(3)}(x)
&=
x+xF_{2,\infty}^{(3)}(x)
+x^2F_{1,\infty}^{(3)}(x)
+2x^3F_{0,\infty}^{(3)}(x)
+xF_{\infty,2}^{(3)}(x) \notag\\
&=
x\left(F_{\infty,2}^{(3)}(x)+1\right)
+x^2F_{1,\infty}^{(3)}(x)
+(x+2x^4)F_{2,\infty}^{(3)}(x)
+2x^4.
\end{align*}
Solving \eqref{eq:m3-infinite-system} and substituting the finite-boundary
states found above gives
\begin{equation}\label{eq:m3gf}
A^{(3)}(x)=1+F_{\infty,\infty}^{(3)}(x)=\frac{N_3(x)}{D_3(x)},
\end{equation}
where
\begin{align}
N_3(x)={}&
1-x-x^2+x^3+4x^4-4x^6-3x^7-3x^8 \notag\\
&{}-5x^9-3x^{10}+2x^{11}+2x^{12},
\end{align}
and
\begin{align}
D_3(x)={}&
1-2x-x^2+x^3+x^4+2x^5+2x^6+2x^7 \notag\\
&{}-4x^8-2x^9+x^{10}-2x^{11}+x^{13}.
\end{align}
The first values of $a_n^{(3)}$ are
\[
1,1,2,5,14,28,55,108,214,412,787,1497,2841,5364,10088,\ldots.
\]
The denominator gives the following recurrence for $n\ge13$:
\begin{align}
a_n^{(3)}={}&
2a_{n-1}^{(3)}+a_{n-2}^{(3)}-a_{n-3}^{(3)}-a_{n-4}^{(3)}
-2a_{n-5}^{(3)}-2a_{n-6}^{(3)}-2a_{n-7}^{(3)} \notag\\
&{}+4a_{n-8}^{(3)}+2a_{n-9}^{(3)}-a_{n-10}^{(3)}
+2a_{n-11}^{(3)}-a_{n-13}^{(3)}.
\end{align}

\section{Exponential growth from cyclic components}
\label{sec:growth}
The purpose of this section is to determine the \(n\)th-root growth of \(a_n^{(m)}\) for fixed \(m\).
For a nonnegative sequence $b=(b_n)_{n\ge0}$, write
\begin{align*}
\underline\alpha(b)&=\liminf_{n\to\infty}b_n^{1/n},\\
\overline\alpha(b)&=\limsup_{n\to\infty}b_n^{1/n}.
\end{align*}
For the present family set
\begin{align*}
\underline\alpha_m&=
\liminf_{n\to\infty}(a_n^{(m)})^{1/n},\\
\overline\alpha_m&=
\limsup_{n\to\infty}(a_n^{(m)})^{1/n}.
\end{align*}
When these two quantities coincide, we denote their common value as
$\alpha_m$ and call it the exponential growth constant for the fixed value of
$m$. 
If $R_m$ denotes the radius of convergence of $A^{(m)}(x)$, then the
Cauchy--Hadamard formula gives
\begin{equation}
\overline\alpha_m=R_m^{-1}.
\end{equation}
Moreover, since the coefficients of $A^{(m)}(x)$ are nonnegative, Pringsheim's
theorem ensures that $x=R_m$ is a singularity of $A^{(m)}(x)$
\cite[Thm.~IV.6]{FlajoletSedgewick2009}. 
Since $A^{(m)}(x)$ is rational, this singularity is a pole.  
We call the poles of modulus $R_m$ dominant poles.

The radius of convergence determines the exponential rate only at the level of the limsup.
To prove a genuine \(n\)th-root limit, we must also exclude other poles on the circle \(|z|=R_m\): such competing dominant poles would contribute terms of the same exponential size and could produce periodic oscillations or cancellations.
Thus the problem is to identify the poles of \(A^{(m)}(x)\) with smallest modulus and to prove that the positive pole is the only pole on that circle.

By Cramer's rule applied to the finite-state system, the denominator of the reduced form of \(F^{(m)}_{\infty,\infty}(x)\) divides \(\det(I-W_m(x))\). 
Hence the dominant poles of \(A^{(m)}(x)=1+F^{(m)}_{\infty,\infty}(x)\) are among the zeros of this determinant.
The block-triangular decomposition of the state system reduces the possible sources of such poles to the cyclic strongly connected components of the dependency graph.
We begin by assigning to each cyclic component \(C\) a positive number
\(r_C\), called its component radius, as the solution of a Perron--Frobenius
spectral-radius equation.  We also define the component growth rate
\(\lambda_C\) as the reciprocal of \(r_C\).

We then study the zeros of the determinants \(\det(I-W_C(x))\) associated with a cyclic component \(C\).
The determinant has no zero in the disk \(|z|<r_C\). 
For the components that can attain the smallest component radius, the only zero on the boundary \(|z|=r_C\) is the positive zero \(z=r_C\).

By verifying that there exists a directed path from each of the cyclic components $\{U_m,V_m,I_m\}$ to the unrestricted state $(\infty, \infty)$ and that $I_m$ is not of minimal component radius, we finally show that the radius of convergence of \(A^{(m)}(x)\) is
\[
        \rho_m=\min\{r_{U_m},r_{V_m}\},
\]
and hence, for \(m\ge2\), the exponential growth constant is
\[
        \alpha_m=\rho_m^{-1}
        =\max\{\lambda_U(m),\lambda_V(m)\}, 
\]
where \(\lambda_U(m)=\lambda_{U_m}\) and \(\lambda_V(m)=\lambda_{V_m}\).

After proving this fixed-\(m\) result, we treat \(m=2\) and \(m=3\) separately.  
In these two cases the dominant pole is shown to be simple, which gives an explicit asymptotic equivalent. 
For a general \(m\), the comparison between \(\lambda_U(m)\) and \(\lambda_V(m)\) is the relevant issue for the simplicity: if the two are distinct, only one of \(U_m\) and \(V_m\) can be critical at the dominant radius, whereas a tie \(\lambda_U(m)=\lambda_V(m)\) would make both components critical at the same radius.
We leave the general comparison of these two rates open.
We finish with estimates uniform in \(m\) and with numerical values of the component growth rates.

\subsection{Component radii and component growth rates}
\label{subsec:component-singularities}

Recall that, for a strongly connected component \(C\) of \(\Gamma_m\), \(W_C(x)\) denotes the corresponding component matrix. 

After the states are ordered by the strongly connected components of
\(\Gamma_m\), the matrix \(W_m(x)\) is block triangular.  Equivalently,
\(I-W_m(x)\) is block triangular with diagonal blocks \(I-W_C(x)\), where
\(C\) runs over the strongly connected components of \(\Gamma_m\).  A
singleton component without a self-loop has \(W_C(x)=(0)\), and hence
contributes the diagonal block \((1)\) to \(I-W_m(x)\).  Thus, at the level
of the finite state system, possible nontrivial determinant factors can come
only from cyclic components.

For \(m\ge2\), \Cref{prop:recurrent-blocks} identifies the cyclic components
as
\[
        U_m,\qquad V_m,\qquad I_m=\{(\infty,m-1)\}.
\]
Consequently,
\[
        \det(I-W_m(x))
        =
        \det(I-W_{U_m}(x))
        \det(I-W_{V_m}(x))
        \det(I-W_{I_m}(x)).
\]

This subsection attaches to each cyclic component \(C\) a positive number \(r_C\), defined by a Perron--Frobenius spectral-radius equation.  
The number \(r_C\) and the reciprocal \(r_C^{-1}\) will be called the component radius and the component growth rate of \(C\), respectively.

Let \(C\) be a cyclic strongly connected component of \(\Gamma_m\).
For a square matrix $B$, let
\[
\spr(B)=\max\{|\lambda|:\lambda\text{ is an eigenvalue of }B\}
\]
be its spectral radius.

For $x>0$, the matrix $W_C(x)$ is nonnegative and irreducible.  With the edge
orientation used in \Cref{subsec:dependency-graph}, an entry
$(W_C(x))_{d,c}$ gives an edge $c\to d$.  Some matrix texts use the opposite
orientation for the directed graph of a matrix, but reversing all arrows does
not change strong connectivity.  Since every nonzero polynomial entry of
$W_C(x)$ is positive for $x>0$, the directed graph of $W_C(x)$ is strongly
connected, and hence $W_C(x)$ is irreducible.

We shall use the following finite-dimensional form of the Perron--Frobenius
theorem.

\begin{lemma}[Perron--Frobenius theorem]
\label{lem:perron-frobenius}
Let $B$ be a nonzero nonnegative irreducible square matrix, and let
$\rho=\spr(B)$.  
Then $\rho>0$ is an algebraically simple eigenvalue of $B$.  
There exist vectors $v>0$ and $u>0$ such
that
\[
        Bv=\rho v,\qquad u^\top B=\rho u^\top.
\]
These positive right and left Perron vectors are unique up to scalar
multiplication.  Moreover, every nonzero nonnegative right or left eigenvector
associated with $\rho$ is strictly positive.
\end{lemma}

\begin{proof}
This is the standard Perron--Frobenius theorem for irreducible nonnegative
matrices; see, for example, Horn and Johnson \cite[Ch.~8]{HornJohnson2012}.
\end{proof}

\begin{lemma}[Component radius of a cyclic component]
\label{lem:component-singularity}
Let $C$ be a cyclic strongly connected component of $\Gamma_m$, and set
\[
\phi_C(x)=\spr(W_C(x))\qquad(x>0).
\]
Then there is a unique positive real number $r_C\in\left(0,1\right]$ such that
\[
\phi_C(r_C)=1.
\]
Moreover, $r_C$ is the smallest positive zero of
\[
\det(I-W_C(x)).
\]
\end{lemma}

\begin{proof}
The function $\phi_C(x)$ is continuous, because the entries of $W_C(x)$ depend continuously on $x$, and the eigenvalues of a finite matrix depend continuously on its entries \cite[Appendix~D]{HornJohnson2012}.  
All entries of $W_C(x)$ are polynomials in $x$ with nonnegative coefficients and zero constant term.  
Therefore $W_C(x)$ tends to the zero matrix as $x\downarrow0$, and hence $\phi_C(x)\to0$.  
If $0<x<y$, then $W_C(x)\le W_C(y)$
entrywise, and at least one entry is strictly smaller at $x$ than at $y$.
Let $v>0$ be a right Perron vector of $W_C(x)$, and let $u>0$ be a left Perron
vector of $W_C(y)$.  Then
\[
\phi_C(y)u^\top v
=u^\top W_C(y)v
>u^\top W_C(x)v
=\phi_C(x)u^\top v.
\]
Since $u^\top v>0$, we have $\phi_C(x)<\phi_C(y)$.  Thus $\phi_C$ is strictly
increasing on $(0,\infty)$.

It remains to show that $\phi_C(1)\ge1$.
Since $W_C(1)$ is a nonzero irreducible nonnegative integer matrix, no row of $W_C(1)$ is zero. 
Hence
\[
W_C(1)\mathbf{1}_C\ge\mathbf{1}_C.
\]
Let $u>0$ be a left Perron vector of $W_C(1)$. 
Then
\begin{align*}
\phi_C(1)u^\top\mathbf{1}_C
&=u^\top W_C(1)\mathbf{1}_C\\
&\ge u^\top\mathbf{1}_C.
\end{align*}
Since $u^\top\mathbf{1}_C>0$, we obtain $\phi_C(1)\ge1$.
Since $\phi_C$ is continuous and strictly increasing, there is a unique $r_C\in\left(0,1\right]$ satisfying $\phi_C(r_C)=1$.

At $x=r_C$, \Cref{lem:perron-frobenius} shows that $1$ is an eigenvalue of
$W_C(r_C)$.  Hence $\det(I-W_C(r_C))=0$.  Conversely, if $0<x<r_C$ and
$\det(I-W_C(x))=0$, then $1$ is an eigenvalue of $W_C(x)$.  Therefore
$\spr(W_C(x))\ge1$, contradicting $\phi_C(x)<\phi_C(r_C)=1$.  Thus
$r_C$ is the smallest positive zero of $\det(I-W_C(x))$.
\end{proof}

For a cyclic component $C$, we call the number $r_C$ from
\Cref{lem:component-singularity} the component radius of $C$.  We also set
\begin{equation}
        \lambda_C=r_C^{-1}
\end{equation}
and call it the component growth rate of $C$.
For $m\ge2$, \Cref{prop:recurrent-blocks} gives three cyclic components:
$V_m$, $U_m$, and $I_m=\{(\infty,m-1)\}$. 
For \(m\ge2\), we write
\[
        \lambda_U(m)=\lambda_{U_m}=r_{U_m}^{-1},
        \qquad
        \lambda_V(m)=\lambda_{V_m}=r_{V_m}^{-1}.
\]
As will be shown later in \Cref{thm:general-growth-constant}, the radius of
convergence of \(A^{(m)}(x)\) is identified with
\[
        \min\{r_{U_m},r_{V_m}\}.
\]
The next lemma gives the preliminary component comparison needed for this reduction: it removes the singleton component \(I_m\) from the set of possible minimal-radius components and also handles the exceptional component \(V_2\).

\begin{lemma}[Preliminary dominance reductions]
\label{lem:preliminary-dominance-reductions}
For every \(m\ge2\), one has
\[
        r_{I_m}=1
        \qquad\text{and}\qquad
        r_{U_m}<1.
\]
Moreover, for \(m=2\), one has \(r_{V_2}=1\).  Consequently, the singleton
component \(I_m\) is never of minimal component radius, and \(V_2\) is not of
minimal component radius when \(m=2\).  Thus any cyclic component of minimal
radius is either \(U_m\), or \(V_m\) with \(m\ge3\).
\end{lemma}

\begin{proof}
The component $I_m$ has component matrix $(x)$, so $r_{I_m}=1$.

For $U_m$, consider the two states $(0,\infty)$ and $(m-1,\infty)$.  At
$x=1$, the corresponding principal submatrix dominates
\[
\begin{pmatrix}
0 & 1\\
C_{m-1} & 1
\end{pmatrix}.
\]
The entry $1$ in the upper right comes from the $1$-split edge, and the entry
$C_{m-1}$ in the lower left comes from the $m$-split edge with coefficient
$c_{m,m-1}=C_{m-1}$, using \Cref{rem:ckp-facts}.  This matrix has spectral
radius
\[
\frac{1+\sqrt{1+4C_{m-1}}}{2}>1.
\]
For nonnegative matrices, the spectral radius is monotone under entrywise
comparison.
Thus, a nonnegative principal submatrix has spectral radius at most that of the whole matrix, after embedding it by zeros into the full index set. 
Hence $\spr(W_{U_m}(1))>1$.  By the strict increase of $x\mapsto\spr(W_{U_m}(x))$, we have $r_{U_m}<1$.

For $m=2$, the component $V_2=\{(0,1),(1,0)\}$ has component matrix
\[
\begin{pmatrix}0&x\\ x&0\end{pmatrix},
\]
so $r_{V_2}=1$.
Since \(r_{U_2}<1\), the component \(V_2\) is not of
minimal component radius.  
The final assertion follows.
\end{proof}

\subsection{Zeros of component determinants}
\label{subsec:weighted-zero-exclusion}

The component radius \(r_C\) was defined using the positive real matrix
\(W_C(r_C)\).  We now show that this number also controls the complex zeros
of the component determinant
\[
        \det(I-W_C(z)).
\]
The first result uses only nonnegativity and excludes zeros in the open disk
\(|z|<r_C\).  This will give analyticity inside the eventual disk of
convergence.

However, controlling the open disk alone is not sufficient for the existence of the \(n\)th-root limit.
We must also rule out competing zeros on the boundary
\(|z|=r_C\) for the components that can attain the minimal radius.  Boundary
zeros are governed by the total weights of closed walks in the dependency
graph, so we introduce a weighted aperiodicity condition.

\begin{lemma}[No zeros inside the component-radius disk]
\label{lem:interior-zero-exclusion}
Let $C$ be a cyclic strongly connected component, and let $r_C$ be its component
radius.  
Then
\begin{equation}
        \det(I-W_C(z))\ne0
        \qquad\text{whenever } |z|<r_C.
\end{equation}
\end{lemma}

\begin{proof}
Throughout this proof, absolute values of matrices are taken entrywise.  
Suppose that $|z|<r_C$ and $\det(I-W_C(z))=0$. 
If $z=0$, this is impossible because $W_C(0)=0$.
Hence $|z|>0$.
Choose $y\ne0$ such that $W_C(z)y=y$. 
Since the entries of $W_C(x)$ have nonnegative
coefficients,
\[
|y|=|W_C(z)y|\le W_C(|z|)|y|.
\]
Let $u>0$ be a left Perron vector of $W_C(|z|)$. Then
\begin{align*}
u^\top|y|
&\le u^\top W_C(|z|)|y|\\
&=\phi_C(|z|)u^\top|y|.
\end{align*}
Since $u^\top|y|>0$, \(\phi_C(|z|)\ge1\).
This contradicts $\phi_C(|z|)<\phi_C(r_C)=1$.
Thus $\det(I-W_C(z))\ne0$ for $|z|<r_C$.
\end{proof}

The preceding lemma excludes zeros strictly inside the component-radius circle. 
It does not, by itself, exclude additional zeros on the boundary.
The periodicity in the weighted dependency graph can produce boundary zeros different from the positive point \(r_C\). 
We now record the weighted period that detects this obstruction.

\begin{definition}[Weighted period of a cyclic component]
\label{def:weighted-period}
Let $C$ be a cyclic component of $\Gamma_m$, and let $W_C(x)$ be the principal
submatrix of $W_m(x)$ indexed by the states in $C$.  The weighted dependency graph
of $C$ is the directed multigraph obtained from the polynomial matrix $W_C(x)$ as
follows: if the coefficient of $x^j$ in the entry $(W_C(x))_{d,c}$ is positive,
then we put an edge $c\to d$ of weight $j$.  Equivalently, a $k$-split edge has
weight $k$, and an append-edge has weight $1$.

The total weight of a closed walk is the sum of the weights of its edges.  The
weighted period of $C$ is the greatest common divisor of the total weights of
all directed closed walks in the weighted dependency graph.  We say that $C$ is
weighted aperiodic if this weighted period is $1$.
\end{definition}

\begin{lemma}[Weighted aperiodicity of $U_m$ and of $V_m$ for $m\ge3$]
\label{lem:weighted-aperiodic}
For every $m\ge2$, the component $U_m$ is weighted aperiodic.  For every
$m\ge3$, the component $V_m$ is weighted aperiodic.
\end{lemma}

\begin{proof}
In $U_m$, the state $(m-1,\infty)$ has a $1$-split self-loop, whose weight is $1$.
Thus the weighted period of $U_m$ is $1$.

Now assume $m\ge3$ and put $b=(0,m-1)\in V_m$.  The component $V_m$ contains the
closed walk
\[
(0,m-1)\to(1,m-1)\to\cdots\to(m-2,m-1)\to(m-1,m-2)\to(0,m-1),
\]
in which every edge has weight $1$.  This closed walk has total weight $m$.
It also contains the closed walk
\[
(0,m-1)\to(1,0)\to(1,m-1)\to\cdots\to(m-2,m-1)\to(m-1,m-2)\to(0,m-1).
\]
The first edge is an append-edge of weight $1$.  The edge
$(1,0)\to(1,m-1)$ is an $(m-1)$-split edge of weight $m-1$; it exists because
$c_{m-1,1}>0$ for $m\ge3$ by \Cref{rem:ckp-facts}.  The remaining part from
$(1,m-1)$ back to $b$ has total weight $m-1$.  Hence the second closed walk has
total weight $2m-1$.  Since $\gcd(m,2m-1)=1$, the weighted period of $V_m$ is
$1$.
\end{proof}

\begin{lemma}[Uniqueness of the boundary zero]
\label{lem:boundary-zero-exclusion}
Let $C$ be a weighted aperiodic cyclic component.  If
\[
\det(I-W_C(z))=0
\qquad\text{and}\qquad
|z|=r_C,
\]
then $z=r_C$.
\end{lemma}

\begin{proof}
Write $z=r_Ce^{i\theta}$, and let $y\ne0$ satisfy $W_C(z)y=y$.  Absolute values
of vectors are taken entrywise.  Then
\[
|y|=|W_C(z)y|\le W_C(r_C)|y|.
\]
Let $u>0$ be a left Perron vector of $W_C(r_C)$.  Since
$\spr(W_C(r_C))=1$, we have $u^\top W_C(r_C)=u^\top$.  Multiplying the preceding
inequality by $u^\top$ gives
\[
u^\top|y|\le u^\top W_C(r_C)|y|=u^\top|y|.
\]
Thus equality holds.  Since $u>0$ and $|y|\le W_C(r_C)|y|$, this implies
\[
|y|=W_C(r_C)|y|.
\]
By irreducibility of $W_C(r_C)$, $|y|$ is strictly positive.

The equality above means that equality holds in the triangle inequality in each
coordinate of $W_C(z)y=y$.  Consider an edge $c\to d$ of weight $j$ in the
weighted dependency graph.  The corresponding summand in the $d$th coordinate has
argument
\[
\arg y_c+j\theta.
\]
Expanding each polynomial entry into its monomial contributions, equality in the
triangle inequality forces all nonzero summands contributing to the $d$th
coordinate to have the same argument as $y_d$, modulo $2\pi$.  Therefore
\[
\arg y_d\equiv \arg y_c+j\theta \pmod {2\pi}
\]
for every weighted edge $c\to d$ of weight $j$.

Summing these congruences along a closed walk of total weight $L$ gives
\[
L\theta\equiv0\pmod {2\pi}.
\]
If $C$ is weighted aperiodic, the greatest common divisor of all such $L$ is
$1$, so $\theta\equiv0\pmod {2\pi}$.  Hence $z=r_C$.
\end{proof}

\subsection{Output-accessibility}
\label{subsec:output-accessibility}

The component radii are properties of cyclic components in the state
system.  
To contribute to the unrestricted generating function $F_{\infty, \infty}^{(m)}(x)$, a component must
also feed into the output state $(\infty,\infty)$.

\begin{definition}[Output-accessible component]
\label{def:output-accessible}
Let $o=(\infty,\infty)$ be the output state.  A cyclic component $C$ of
$\Gamma_m$ is called output-accessible if some vertex of $C$ reaches $o$ by a
directed path in $\Gamma_m$.  Equivalently, there exist $c\in C$ and $t\ge0$ such that
\[
\left(W_m(x)^t\right)_{o,c}\neq0
\]
as a polynomial in $x$.  This is a graph condition saying that, after
back-substitution in the block-triangular system, values from the component $C$
can enter the formula for $F_{\infty,\infty}^{(m)}(x)$.
\end{definition}

\begin{lemma}[Output-accessibility in the endpoint system]
\label{lem:output-accessible-components}
For $m\ge2$, all three cyclic components $U_m,V_m,I_m$ of $\Gamma_m$ are
output-accessible.
\end{lemma}

\begin{proof}
The component $U_m$ reaches $o$ directly through the terms
\[
\sum_{k=1}^m c_{k,\infty}x^kF_{m-k,\infty}^{(m)}(x)
\]
in the row for $F_{\infty,\infty}^{(m)}(x)$.  The singleton component
$I_m=\{(\infty,m-1)\}$ reaches $o$ through the term
$xF_{\infty,m-1}^{(m)}(x)$.  Finally, $V_m$ reaches $U_m$, for instance by the
append-edge
\[
(0,m-1)\to(1,\infty),
\]
and $U_m$ reaches $o$.
\end{proof}

\subsection{The exponential growth constant for fixed \texorpdfstring{$m$}{m}}
\label{subsec:fixed-m-growth}

We now assemble the preceding component analysis.  Set
\[
        \rho_m=\min\{r_{U_m},r_{V_m}\}.
\]
By \Cref{lem:preliminary-dominance-reductions}, this is the minimum among all
cyclic component radii of \(\Gamma_m\).  
By \Cref{lem:interior-zero-exclusion}, all cyclic diagonal blocks of \(I-W_m(z)\) are invertible for \(|z|<\rho_m\); the acyclic singleton blocks are equal to \((1)\). 
Hence \(A^{(m)}(z)\) is analytic in this disk.

It remains to prove that the positive point \(z=\rho_m\) is an actual
singularity of the unrestricted generating function, rather than a component
zero that disappears after solving the full block-triangular system.  The next
lemma proves this using nonnegativity and
output-accessibility.

\begin{lemma}[Propagation of a minimal component pole]
\label{lem:output-pole-propagation}
Assume $m\ge2$.  Let $C$ be an output-accessible cyclic component of $\Gamma_m$
whose component radius $r_C$ is minimal among all cyclic component radii of
$\Gamma_m$.  Then $A^{(m)}(x)$ has a pole at $x=r_C$.
\end{lemma}

\begin{proof}
For $0<x<r_C$, all state generating functions are finite: after ordering by
strongly connected components, each cyclic diagonal block is invertible by
\Cref{lem:interior-zero-exclusion}, and the acyclic singleton components
contribute diagonal blocks equal to $(1)$.

For such $x$, the component equation for $C$ has the form
\[
F_C(x)=x\mathbf 1_C+W_C(x)F_C(x)+G_C(x),
\]
where $\mathbf 1_C$ is the all-one vector indexed by the states in $C$, and
$G_C(x)$ is a vector with nonnegative entries collecting the contributions from
states outside $C$.

Let $\phi_C(x)=\spr(W_C(x))$.  Then $\phi_C(x)<1$ for $0<x<r_C$, and
$\phi_C(x)\uparrow1$ as $x\uparrow r_C$.  Let $u(x)>0$ be a left Perron vector of
$W_C(x)$ and normalize it by $u(x)^\top\mathbf 1_C=1$.  Multiplying the component
equation by $u(x)^\top$ gives
\[
(1-\phi_C(x))u(x)^\top F_C(x)
=x+u(x)^\top G_C(x)
\ge x.
\]
Therefore
\[
u(x)^\top F_C(x)\ge \frac{x}{1-\phi_C(x)}\to\infty
\qquad(x\uparrow r_C).
\]
Since the coordinates of $F_C(x)$ are nonnegative and
$u(x)^\top\mathbf 1_C=1$, we have
\[
\max_{c\in C}F_c(x)\to\infty
\qquad(x\uparrow r_C).
\]
As $C$ is finite, at least one coordinate $F_c(x)$ is unbounded; since each
coordinate is a nondecreasing power series with nonnegative coefficients, that
coordinate tends to $+\infty$ as $x\uparrow r_C$.

If $c\to d$ is an edge of $\Gamma_m$, then $W_{d,c}(x)$ is a nonzero polynomial
with nonnegative coefficients, so $W_{d,c}(r_C)>0$.  From the state equation,
\[
F_d(x)=x+\sum_e W_{d,e}(x)F_e(x)\ge W_{d,c}(x)F_c(x).
\]
Thus divergence propagates along directed edges. 
Let \(c_0\in C\) be a coordinate for which \(F_{c_0}(x)\to\infty\).
Since \(C\) is strongly connected, \(c_0\) reaches every vertex of \(C\).
By output-accessibility, some vertex of \(C\) reaches \(o=(\infty,\infty)\).
Hence \(c_0\) reaches \(o\).
Repeating the preceding propagation argument along this path gives
\[
F_{\infty,\infty}^{(m)}(x)\to\infty
\qquad(x\uparrow r_C).
\]
Hence $A^{(m)}(x)=1+F_{\infty,\infty}^{(m)}(x)$ is singular at $x=r_C$.  Since
$A^{(m)}(x)$ is rational, the singularity is a pole.
\end{proof}

To translate a unique dominant pole of the rational generating function into
coefficient information, we use the following standard estimate.

\begin{lemma}[Dominant pole estimate for rational functions]
\label{lem:rational-pole-estimate}
Let $f(z)$ be a rational function analytic at $0$, and let
$f_n=[z^n]f(z)$.  Suppose that $\rho$ is the unique pole of $f$ of smallest
modulus.  If the pole at $\rho$ has order $s$, then there is a nonzero polynomial
$\Pi(n)$ of degree $s-1$.  If $f$ has another pole, let $\Lambda$ be the second
smallest modulus among the poles of $f$; then there is an integer $r\ge0$ such
that
\begin{equation}
        f_n=\Pi(n)\rho^{-n}+O(\Lambda^{-n}n^r).
\end{equation}
If $\rho$ is the only pole of $f$, the error term is absent.
\end{lemma}

\begin{proof}
This is the standard partial-fraction estimate for rational functions; see
Theorem~IV.9 and the following discussion in \cite{FlajoletSedgewick2009}.
\end{proof}

We can now prove that the growth constant exists for every fixed \(m\).
The proof has three steps: first, \(\rho_m\) is the radius of convergence; second, the only pole on the circle \(|z|=\rho_m\) is the positive pole \(z=\rho_m\); and third, the rational-function estimate gives the \(n\)th-root limit.

\begin{theorem}[Growth constant for every fixed adjacency bound]
\label{thm:general-growth-constant}
For $m=1$,
\begin{equation}
        \lim_{n\to\infty}\left(a_n^{(1)}\right)^{1/n}=1.
\end{equation}
For every $m\ge2$, the limit
\[
\lim_{n\to\infty}\left(a_n^{(m)}\right)^{1/n}
\]
exists and is given by
\begin{equation}
        \alpha_m
        =\max\{\lambda_U(m),\lambda_V(m)\}.
\end{equation}
Equivalently, if
\begin{equation}
        \rho_m=\min\{r_{U_m},r_{V_m}\},
\end{equation}
then the radius of convergence of $A^{(m)}(x)$ is $\rho_m$, and
$\alpha_m=\rho_m^{-1}$.
Moreover, \(x=\rho_m\) is the unique dominant
pole of \(A^{(m)}(x)\).
\end{theorem}

\begin{proof}
The case $m=1$ follows from \eqref{eq:m1-genfun}: the sequence is eventually
constant equal to $2$, so its $n$th-root limit is $1$.

Assume $m\ge2$.  By \Cref{lem:preliminary-dominance-reductions}, $r_{I_m}=1$ and
$r_{U_m}<1$.  Hence
\[
\rho_m=\min\{r_{U_m},r_{V_m}\}<1
\]
is the minimum radius among all cyclic components of $\Gamma_m$.

We first show that $A^{(m)}(z)$ is analytic for $|z|<\rho_m$.  After ordering by
strongly connected components, $I-W_m(z)$ is block triangular.  The acyclic
singleton components contribute diagonal blocks equal to $(1)$.  By
\Cref{lem:interior-zero-exclusion}, each cyclic component matrix $I-W_C(z)$ is
invertible whenever $|z|<r_C$.  Since $\rho_m$ is no larger than any cyclic
component radius, all diagonal blocks are invertible for $|z|<\rho_m$.  Thus
$I-W_m(z)$ is invertible there, and all state generating functions are analytic
for $|z|<\rho_m$.  In particular, the radius of convergence $R_m$ of
$A^{(m)}(x)$ satisfies
\[
R_m\ge \rho_m.
\]

Choose $C\in\{U_m,V_m\}$ with $r_C=\rho_m$.  By
\Cref{lem:output-accessible-components}, this component is output-accessible.
Since $r_C$ is minimal among all cyclic component radii,
\Cref{lem:output-pole-propagation} shows that $A^{(m)}(x)$ has a pole at
$x=\rho_m$.  Hence
\[
R_m\le \rho_m.
\]
Combining the two inequalities gives $R_m=\rho_m$.

It remains to show that the full $n$th-root limit exists, not only the limsup.
For every cyclic component with radius larger than $\rho_m$,
\Cref{lem:interior-zero-exclusion} gives no zeros on $|z|=\rho_m$.  A component
with radius equal to $\rho_m$ is one of $U_m$ or $V_m$.  The component $U_m$ is
weighted aperiodic by \Cref{lem:weighted-aperiodic}.  
If $V_m$ has radius $\rho_m$, then $m\ge3$, because in the $m=2$ case \Cref{lem:preliminary-dominance-reductions} gives $r_{V_2}=1>r_{U_2}=\rho_2$, so $V_2$ is not dominant.
For $m\ge3$, $V_m$ is weighted aperiodic by \Cref{lem:weighted-aperiodic}.
Therefore, by \Cref{lem:boundary-zero-exclusion}, the only component zero on $|z|=\rho_m$ is the positive point $z=\rho_m$ itself.

Thus $A^{(m)}(x)$ has a unique dominant pole, located at $x=\rho_m$.
\Cref{lem:rational-pole-estimate} gives
\begin{equation}
        a_n^{(m)}=\Pi(n)\rho_m^{-n}+O(\Lambda^{-n}n^r)
\end{equation}
for some $\Lambda>\rho_m$, some integer $r\ge0$, and a nonzero real polynomial
$\Pi$.  The error term is exponentially smaller than the first term.  Since
$a_n^{(m)}\ge0$ and $\Pi$ has an eventual sign, this sign must be positive.
Therefore
\begin{equation}
        \lim_{n\to\infty}\left(a_n^{(m)}\right)^{1/n}=\rho_m^{-1}.
\end{equation}
Since $\rho_m^{-1}=\max\{\lambda_U(m),\lambda_V(m)\}$, the theorem follows.
\end{proof}

If $\rho_m$ is a simple pole of $A^{(m)}(x)$, we can further specify the asymptotic form of $a_n^{(m)}$.
\begin{corollary}[Simple dominant pole asymptotics]
\label{cor:simple-pole-asymptotics}
Suppose that $A^{(m)}(x)=N_m(x)/D_m(x)$ is written in reduced form, and suppose that $\rho_m$ is a simple pole.
Then
\begin{equation}
        a_n^{(m)}\sim \kappa_m\rho_m^{-n}
        =\kappa_m\alpha_m^n,
        \qquad
        \kappa_m=-\frac{N_m(\rho_m)}{\rho_mD_m'(\rho_m)}.
\end{equation}
More precisely, if $\Lambda_m$ is the modulus of the next pole, then
\begin{equation}
        a_n^{(m)}=\kappa_m\rho_m^{-n}
        \left(1+O\left(\left(\frac{\rho_m}{\Lambda_m}\right)^n n^r\right)\right)
\end{equation}
for some integer $r\ge0$.
\end{corollary}

\begin{proof}
This is \Cref{lem:rational-pole-estimate} together with the residue formula
\[
\lim_{x\to\rho_m}(x-\rho_m)A^{(m)}(x)
=\frac{N_m(\rho_m)}{D_m'(\rho_m)}.
\]
\end{proof}

\subsection{Simple-pole asymptotics for the case \texorpdfstring{$m=2$}{m=2}}
\label{subsec:growth-m2}

By \Cref{lem:preliminary-dominance-reductions}, we have
\[
        r_{V_2}=r_{I_2}=1
        \qquad\text{and}\qquad
        r_{U_2}<1.
\]
Hence the minimal component radius is \(r_{U_2}\).  With the order
\[
        (0,\infty),(1,\infty)
\]
for \(U_2\), the component matrix is
\[
        W_{U_2}(x)=
        \begin{pmatrix}
        0 & x\\
        x^2 & x
        \end{pmatrix}.
\]
Thus
\[
        \det(I-W_{U_2}(x))=1-x-x^3.
\]
Let
\[
        p_2(x)=1-x-x^3.
\]
Since \(p_2'(x)=-1-3x^2<0\) for every real \(x\), and since
\(p_2(0)>0\) and \(p_2(1)<0\), the polynomial \(p_2\) has a unique zero in
\((0,1)\).  We denote this zero by \(\rho_2\).  Moreover,
\(p_2(1/2)>0\), so \(\rho_2>1/2\).  Numerically,
\[
        \rho_2\approx 0.682.
\]

Therefore \(r_{U_2}=\rho_2\).  By
\Cref{thm:general-growth-constant}, \(x=\rho_2\) is the unique dominant pole
of \(A^{(2)}(x)\).  It remains only to check that this pole is simple.

From \eqref{eq:m2gf}, the denominator is
\[
        D_2(x)=(1-x)^2p_2(x).
\]
Since \(p_2'(\rho_2)\ne0\) and \(\rho_2\ne1\), the denominator has a simple
zero at \(x=\rho_2\).  Because \(x=\rho_2\) is an actual pole by
\Cref{thm:general-growth-constant}, it is a simple pole of \(A^{(2)}(x)\).

The numerator in \eqref{eq:m2gf} is
\[
        N_2(x)=1-2x+2x^2-x^4-x^6.
\]
Using \(\rho_2^3=1-\rho_2\), we obtain
\[
        N_2(\rho_2)=\rho_2(2\rho_2-1).
\]
Also
\[
        D_2'(\rho_2)
        =(1-\rho_2)^2p_2'(\rho_2)
        =-(1-\rho_2)^2(1+3\rho_2^2).
\]
Hence
\[
        \kappa_2
        =
        -\frac{N_2(\rho_2)}{\rho_2D_2'(\rho_2)}
        =
        \frac{2\rho_2-1}{(1-\rho_2)^2(1+3\rho_2^2)}
        \approx 1.51.
\]
By \Cref{cor:simple-pole-asymptotics},
\[
        a_n^{(2)}\sim \kappa_2\alpha_2^n,
        \qquad
        \alpha_2=\rho_2^{-1}\approx 1.47.
\]

\subsection{Simple-pole asymptotics for the case \texorpdfstring{$m=3$}{m=3}}
\label{subsec:growth-m3}

For \(m=3\), the cyclic components with more than one vertex are
\[
        V_3=\{(0,2),(1,0),(1,2),(2,0),(2,1)\}
\]
and
\[
        U_3=\{(0,\infty),(1,\infty),(2,\infty)\}.
\]
With the order
\[
        (0,2),(1,0),(1,2),(2,0),(2,1)
\]
for \(V_3\), its component matrix is
\[
        W_{V_3}(x)=
        \begin{pmatrix}
        0 & 0 & 0 & 0 & x\\
        x & 0 & 0 & 0 & 0\\
        x & x^2 & 0 & 0 & x\\
        0 & 0 & x & 0 & 0\\
        0 & 0 & x & x & 0
        \end{pmatrix},
\]
and
\[
        \det(I-W_{V_3}(x))
        =
        1-x^2-2x^3-x^4-x^5-x^6.
\]
With the order
\[
        (0,\infty),(1,\infty),(2,\infty)
\]
for \(U_3\), its component matrix is
\[
        W_{U_3}(x)=
        \begin{pmatrix}
        0 & 0 & x\\
        x^3 & x^2 & x\\
        2x^3 & x^2 & x
        \end{pmatrix},
\]
and
\[
        \det(I-W_{U_3}(x))
        =
        (1+x)(x^5-x^4-x^3+x^2-2x+1).
\]

Let
\[
        H(x)=x^5-x^4-x^3+x^2-2x+1
\]
and
\[
        G(x)=x^6+x^5+x^4+2x^3+x^2-1.
\]

We first locate the component radius of \(U_3\).  For \(0\le x\le1\), we have
\(5x^4-4x^3\le x^3\) and \(x(1-x)(2-x)\le1/2\).  Hence
\[
        H'(x)
        =
        5x^4-4x^3-3x^2+2x-2
        \le
        x^3-3x^2+2x-2
        =
        x(1-x)(2-x)-2
        \le
        -\frac32<0.
\]
Thus \(H\) is strictly decreasing on \([0,1]\).  Since
\(H(0.54)>0\) and \(H(0.55)<0\), it has a unique zero in
\((0.54,0.55)\).  We denote this zero by \(\rho_3\).  Then
\[
        r_{U_3}=\rho_3,
        \qquad
        \rho_3\approx0.547.
\]

We next compare this with the radius of \(V_3\). 
Since \(G(0)=-1\), \(G(1)>0\), and
\[
        G'(x)=6x^5+5x^4+4x^3+6x^2+2x>0
\]
for \(x>0\), the polynomial \(G\) has a unique positive zero.
Moreover, because \(\rho_3<0.55\) and \(G(0.55)<0\), we have
\[
        G(\rho_3)<0.
\]
Therefore the positive zero of \(G\), namely \(r_{V_3}\), satisfies
\[
        r_{V_3}>\rho_3=r_{U_3}.
\]
Since also \(r_{I_3}=1>\rho_3\), the minimal component radius is \(r_{U_3}\).

By \Cref{thm:general-growth-constant}, \(x=\rho_3\) is the unique dominant
pole of \(A^{(3)}(x)\).  It remains only to check that this pole is simple.

From \eqref{eq:m3gf}, the denominator factors as
\[
        D_3(x)
        =
        (x-1)(x+1)H(x)G(x).
\]
We have \(H'(\rho_3)<0\), and \(G(\rho_3)<0\).  Also
\(\rho_3\ne\pm1\).  Hence \(D_3\) has a simple zero at \(x=\rho_3\).
Since \(x=\rho_3\) is an actual pole by
\Cref{thm:general-growth-constant}, it is a simple pole of \(A^{(3)}(x)\).

Therefore, by \Cref{cor:simple-pole-asymptotics},
\[
        \kappa_3
        =
        -\frac{N_3(\rho_3)}{\rho_3D_3'(\rho_3)}
        \approx 2.99,
\]
and
\[
        a_n^{(3)}\sim \kappa_3\alpha_3^n,
        \qquad
        \alpha_3=\rho_3^{-1}\approx1.83.
\]

\subsection{Dependence on \texorpdfstring{$m$}{m} and the limit \texorpdfstring{$m\to\infty$}{m to infinity}}
\label{subsec:large-m-rates}

The fixed-$m$ theorem identifies the growth constant through the cyclic
components. 
We now prove estimates that hold uniformly as $m$ varies by direct comparison and a constructive lower bound. 
These estimates explain two global features: the growth constants are monotone in the
adjacency bound $m$, and they approach $4$ as $m\to\infty$.

\begin{theorem}[Large-adjacency estimates]
\label{thm:large-m-rates}
For every $m\ge1$, the exponential growth constant $\alpha_m$ exists.  Moreover,
\begin{equation}
        \alpha_m\le\alpha_{m+1},
\end{equation}
for every $m\ge1$, and
\begin{equation}
        C_{m-1}^{1/(m+1)}\le \alpha_m<4
\end{equation}
for every finite $m$.  Consequently,
\begin{equation}
        \lim_{m\to\infty}\alpha_m=4.
\end{equation}
More precisely,
\begin{equation}
        4-\alpha_m=O\left(\frac{\log m}{m}\right)
        \qquad(m\to\infty).
\end{equation}
\end{theorem}

\begin{proof}
The existence of $\alpha_m$ is \Cref{thm:general-growth-constant}.  The
monotonicity in $m$ follows from containment. 
Indeed,
\[
\A_n^{(m)}\subseteq \A_n^{(m+1)}
\]
for every $n$, and hence $a_n^{(m)}\le a_n^{(m+1)}$.  Taking $n$th roots and
then limits gives $\alpha_m\le\alpha_{m+1}$.

We next prove the strict upper bound $\alpha_m<4$ for every finite $m$.  The
inclusion $\A_n^{(m)}\subseteq\Av_n(132)$ gives only $\alpha_m\le4$, so we use
the finite-state system.  The case $m=1$ is immediate from $\alpha_1=1$; assume
$m\ge2$.

Let $C$ be either $U_m$ or $V_m$.  For a row $(p,q)\in C$ of $W_C(1/4)$, set
\[
R_{p,q}^C=\sum_{(r,s)\in C}(W_C(1/4))_{(p,q),(r,s)}.
\]
We shall use the standard Catalan generating function \cite{FlajoletSedgewick2009}
\[
\operatorname{Cat}(t)=\sum_{j\ge0} C_jt^j
=\frac{1-\sqrt{1-4t}}{2t}
\qquad (0< t<1/4).
\]
Since the coefficients are nonnegative, monotone convergence gives
\[
\sum_{j\ge0} C_j4^{-j}
=
\lim_{t\uparrow 1/4}\sum_{j\ge0}C_jt^j
=
\lim_{t\uparrow 1/4}\operatorname{Cat}(t)
=2.
\]
Since $W_C$ is a principal submatrix of $W_m$ with nonnegative entries,
\begin{align*}
R_{p,q}^C
&\le \sum_{(r,s)\in B_m^2} W^{(m)}_{(p,q),(r,s)}(1/4) \\
&\le \sum_{k=1}^m c_{k,p}4^{-k}+\frac14 \\
&\le \sum_{k=1}^m C_{k-1}4^{-k}+\frac14
 <\frac14\sum_{j=0}^\infty C_j4^{-j}+\frac14
 =\frac34.
\end{align*}
Here we used \Cref{rem:ckp-facts}, and the term $1/4$ accounts for the possible
append-edge contribution.  Hence $\|W_C(1/4)\|_\infty<3/4$, and therefore
\[
\spr(W_C(1/4))\le \|W_C(1/4)\|_\infty<1.
\]
Thus $r_C>1/4$ and $\lambda_C<4$ for $C=U_m,V_m$.  By
\Cref{thm:general-growth-constant}, $\alpha_m=\max\{\lambda_U(m),\lambda_V(m)\}<4$.

It remains to prove the lower bound and the limit.  We construct a subset of
$\A_n^{(m)}$.  Let $L=m+1$ and $k=\lfloor n/L\rfloor$.  Starting from the
decreasing permutation $n,n-1,\ldots,1$, partition the first $kL$ entries into
the consecutive blocks
\[
B_i=(n-(i-1)L,n-(i-1)L-1,\ldots,n-iL+1)
\]
for $1\le i\le k$.  If $r=n-kL>0$, leave the remaining entries
\[
T=(r,r-1,\ldots,1)
\]
in decreasing order.

For each full block $B_i$, keep its first and last entries fixed and permute the
middle $m-1$ entries by an arbitrary $132$-avoiding permutation of those entries.
There are $C_{m-1}$ choices for each block.  These choices are made
independently for the $k$ full blocks, while the final tail $T$, if present, is
left decreasing.

Every permutation obtained in this way belongs to $\A_n^{(m)}$.  The adjacency
constraint is preserved because the entries inside a full block form a set of
$m+1$ consecutive integers, so any adjacent difference inside the block is at
most $m$.  The difference across two consecutive full blocks is $1$, and the
same holds at the boundary between the last full block and the tail, if the tail
is present.  The tail itself is decreasing.

The resulting permutation also avoids $132$.  Inside each full block, the first
entry is the maximum of the block and the last entry is the minimum of the
block.  These two fixed endpoints cannot participate in a $132$-pattern within
that block, and the middle entries were chosen to avoid $132$.  The tail is
decreasing.  Across different parts, every entry in an earlier part is larger
than every entry in a later part.  Thus, if indices $i<j<k$ are not all in the
same part, then $\pi_i$ and $\pi_k$ lie in different parts and $\pi_i>\pi_k$,
which is incompatible with $\pi_i<\pi_k<\pi_j$.  Hence no cross-part
$132$-pattern is created.

Thus
\[
a_n^{(m)}\ge C_{m-1}^{\lfloor n/(m+1)\rfloor}.
\]
Taking $n$th roots and then limits gives
\[
\alpha_m\ge C_{m-1}^{1/(m+1)}.
\]
The standard Catalan asymptotic \cite{FlajoletSedgewick2009}
\[
C_{m-1}\sim \frac{4^{m-1}}{\sqrt{\pi(m-1)^3}}
\]
implies
\[
C_{m-1}^{1/(m+1)}\to4.
\]
Since $C_{m-1}^{1/(m+1)}\le\alpha_m<4$, we obtain $\alpha_m\to4$.

The same asymptotic gives the stated rate of convergence.  Taking logarithms,
\[
\log4-\log C_{m-1}^{1/(m+1)}
=O\left(\frac{\log m}{m}\right).
\]
Therefore
\[
4-C_{m-1}^{1/(m+1)}
=O\left(\frac{\log m}{m}\right).
\]
Since $C_{m-1}^{1/(m+1)}\le\alpha_m<4$, it follows that
\[
0<4-\alpha_m
\le 4-C_{m-1}^{1/(m+1)}
=O\left(\frac{\log m}{m}\right).
\]
This proves the theorem.
\end{proof}

\subsection{Numerical growth constants for selected \texorpdfstring{$m$}{m}}
\label{subsec:numerical-growth}

For $m\ge2$, \Cref{thm:general-growth-constant} gives
\[
\alpha_m=\max\{\lambda_U(m),\lambda_V(m)\}.
\]
The rates are computed by solving the spectral-radius equations
\begin{equation}
\label{eq:component-rate-equations}
\spr(W_{U_m}(r_U))=1,
\qquad
\spr(W_{V_m}(r_V))=1,
\end{equation}
and then setting
\[
\lambda_U(m)=r_U^{-1},
\qquad
\lambda_V(m)=r_V^{-1}.
\]
\Cref{tab:growth-estimates} records selected values up to $m=100$.  The table is
included to show the approach to the Catalan growth constant $4$; the approach
is visible but slow.

The component $V_m$ has
\[
|V_m|=\frac{(m-1)(m+2)}{2}
\]
states.  A dense eigenvalue computation on this component would scale cubically
in $|V_m|$, hence as $O(m^6)$.  The matrices used here are sparse.  The displayed
values were computed by treating $W_{U_m}(x)$ and $W_{V_m}(x)$ as sparse
matrices and solving \eqref{eq:component-rate-equations} by bisection in the
positive variable $r$, using sparse Perron eigenvalue computations at each step.

\begin{table}[H]
\centering
\caption{Component growth rates from the cyclic components $U_m$ and $V_m$.  By
\Cref{thm:general-growth-constant}, the column $\alpha_m$ is the exponential
growth constant.  Values are rounded to three decimal places.  The last column
is the lower bound from \Cref{thm:large-m-rates}.}
\label{tab:growth-estimates}
\begin{tabular}{ccccc}
\toprule
$m$ & $\lambda_U(m)$ & $\lambda_V(m)$ & $\alpha_m$ & $C_{m-1}^{1/(m+1)}$\\
\midrule
2   & 1.466 & 1.000 & 1.466 & 1.000\\
3   & 1.827 & 1.691 & 1.827 & 1.189\\
4   & 2.100 & 2.091 & 2.100 & 1.380\\
5   & 2.312 & 2.352 & 2.352 & 1.552\\
6   & 2.480 & 2.536 & 2.536 & 1.706\\
7   & 2.615 & 2.675 & 2.675 & 1.841\\
8   & 2.728 & 2.786 & 2.786 & 1.961\\
9   & 2.822 & 2.876 & 2.876 & 2.068\\
10  & 2.902 & 2.953 & 2.953 & 2.164\\
20  & 3.333 & 3.357 & 3.357 & 2.756\\
50  & 3.676 & 3.682 & 3.682 & 3.339\\
100 & 3.817 & 3.819 & 3.819 & 3.614\\
\bottomrule
\end{tabular}
\end{table}

For the displayed values, the larger component growth rate is $\lambda_U(m)$ for
$m=2,3,4$ and $\lambda_V(m)$ for the displayed values $m\ge5$.  We do not prove
in this paper that this comparison pattern holds for all larger $m$.  This
question is separated from the existence of the growth constant: the growth
constant is already known from \Cref{thm:general-growth-constant}, but deciding
which of the two component growth rates is larger remains a separate problem.

\section{Connections and further directions}\label{sec:connections}

\paragraph{Bounded gaps and transfer matrices.}
The condition $|\pi_{i+1}-\pi_i|\le m$ is the bounded-gap or $m$-bounded
condition studied in several contexts.  Avgustinovich and Kitaev
\cite{AvgustinovichKitaev2008} related bounded gaps, after inversion and a
parameter shift, to uniquely $k$-determined permutations and transfer matrices.
Gillespie, Monks, and Monks
\cite{GillespieMonksMonks2020} proved rationality for anchored bounded-gap
permutations.  Our method gives rationality in a different class: bounded-gap
permutations that also satisfy the Catalan restriction of $132$-avoidance.  The
maximum decomposition of $132$-avoiders replaces a transfer graph based on
adjacent local states with a graph whose vertices are two endpoint-deficiency
thresholds.

\paragraph{Enumeration schemes.}
Finite-state enumeration of permutation classes is often studied through
generating trees, insertion encodings, and enumeration schemes.  Enumeration
schemes were introduced by Zeilberger \cite{Zeilberger1998} and extended by
Vatter \cite{Vatter2008}; insertion encodings were introduced by
Albert, Linton, and Ru\v{s}kuc \cite{AlbertLintonRuskuc2005}.  For a modern
overview of permutation classes, see Vatter \cite{Vatter2015}.  The present
family is not a permutation class in the usual sense, because the adjacency
condition is not closed under taking arbitrary subpermutations.  Nevertheless,
the endpoint-state recursion gives a finite linear representation for each fixed
$m$.

\paragraph{Catalan objects.}
The class of $132$-avoiders has standard bijections with Dyck paths and with
binary plane trees \cite{Krattenthaler2001,Bona2012Surprising}.  A concrete problem is
to translate the adjacency condition $|\pi_{i+1}-\pi_i|\le m$ into those two
models.  One should also ask whether the endpoint thresholds $(p,q)$ correspond
to specified path or tree boundary statistics, and whether this gives a smaller
finite state space.

\paragraph{Open problems.}
The most immediate problems are the following.
\begin{enumerate}[label=\textbf{P\arabic*.},leftmargin=2.2em]
\item Determine the minimal number of states required to compute
$A^{(m)}(x)$ for general $m$.
\item Describe the degree, factorization, or dominant root of the denominator of
$A^{(m)}(x)$ uniformly in $m$.
\item Compare the two component growth rates $\lambda_U(m)$ and $\lambda_V(m)$.
Numerically, $\lambda_U(m)>\lambda_V(m)$ for $m=2,3,4$, while
$\lambda_V(m)>\lambda_U(m)$ for all computed values $m\ge5$.  Prove or
disprove this pattern.  More generally, determine when the dominant pole is
simple, and thereby obtain $a_n^{(m)}\sim \kappa_m\alpha_m^n$ for each fixed
$m$.
\item Refine the recursion by statistics such as descents, peaks, inversions,
or endpoint values.
\end{enumerate}

\section*{Use of AI-assisted tools}

During the preparation of this work, the authors used OpenAI's ChatGPT
as an interactive research and writing assistant. 
The tool was used to explore possible research directions, to help formulate candidate finite-state recurrences, to assist with symbolic and numerical verification code, and to help draft and edit parts of the exposition.
The mathematical arguments, computations, references, and final text were reviewed, checked, and revised by the authors, who take full responsibility for the content of the paper.

\bibliographystyle{plain}
\bibliography{references}

@article{AvgustinovichKitaev2008,
  author  = {Avgustinovich, Sergey and Kitaev, Sergey},
  title   = {On Uniquely {$k$}-Determined Permutations},
  journal = {Discrete Mathematics},
  volume  = {308},
  number  = {9},
  pages   = {1500--1507},
  year    = {2008},
  doi     = {10.1016/j.disc.2007.03.079}
}

@book{Bona2012,
  author    = {B{\'o}na, Mikl{\'o}s},
  title     = {Combinatorics of Permutations},
  edition   = {2},
  publisher = {CRC Press},
  address   = {Boca Raton},
  year      = {2012},
  doi       = {10.1201/b12210}
}

@book{HornJohnson2012,
  author    = {Horn, Roger A. and Johnson, Charles R.},
  title     = {Matrix Analysis},
  edition   = {2},
  publisher = {Cambridge University Press},
  address   = {Cambridge},
  year      = {2012}
}

@book{FlajoletSedgewick2009,
  author    = {Flajolet, Philippe and Sedgewick, Robert},
  title     = {Analytic Combinatorics},
  publisher = {Cambridge University Press},
  address   = {Cambridge},
  year      = {2009},
  doi       = {10.1017/CBO9780511801655}
}

@article{GillespieMonksMonks2020,
  author  = {Gillespie, Maria M. and Monks, Kenneth G. and Monks, Kenneth M.},
  title   = {Enumerating Anchored Permutations with Bounded Gaps},
  journal = {Discrete Mathematics},
  volume  = {343},
  number  = {9},
  pages   = {111957},
  year    = {2020},
  doi     = {10.1016/j.disc.2020.111957}
}

@misc{Nadler2026,
  author        = {Nadler, Nathaniel},
  title         = {On {$132$}-Avoiding Permutations with an Adjacency Constraint},
  year          = {2026},
  eprint        = {2604.22135},
  archivePrefix = {arXiv},
  primaryClass  = {math.CO},
  note          = {arXiv:2604.22135 [math.CO]}
}

@article{SimionSchmidt1985,
  author  = {Simion, Rodica and Schmidt, Frank W.},
  title   = {Restricted Permutations},
  journal = {European Journal of Combinatorics},
  volume  = {6},
  number  = {4},
  pages   = {383--406},
  year    = {1985},
  doi     = {10.1016/S0195-6698(85)80052-4}
}

@article{Vatter2008,
  author  = {Vatter, Vincent},
  title   = {Enumeration Schemes for Restricted Permutations},
  journal = {Combinatorics, Probability and Computing},
  volume  = {17},
  number  = {1},
  pages   = {137--159},
  year    = {2008},
  doi     = {10.1017/S0963548307008516}
}

@article{DeSantisEtAl2013,
  author    = {DeSantis, Derek and Field, Rebecca and Hough, Wesley and Jones, Brant and Meissen, Rebecca and Ziefle, Jacob},
  title     = {Permutation Pattern Avoidance and the {Catalan} Triangle},
  journal   = {Missouri Journal of Mathematical Sciences},
  volume    = {25},
  number    = {1},
  pages     = {50--60},
  year      = {2013}
}

@article{Zeilberger1998,
  author  = {Zeilberger, Doron},
  title   = {Enumeration Schemes and, More Importantly, Their Automatic Generation},
  journal = {Annals of Combinatorics},
  volume  = {2},
  pages   = {185--195},
  year    = {1998},
  doi     = {10.1007/BF01608488}
}

@article{AlbertLintonRuskuc2005,
  author  = {Albert, Michael H. and Linton, Steve and Ru{\v{s}}kuc, Nik},
  title   = {The Insertion Encoding of Permutations},
  journal = {Electronic Journal of Combinatorics},
  volume  = {12},
  number  = {1},
  pages   = {Research Paper 47},
  year    = {2005},
  doi     = {10.37236/1944}
}

@incollection{Vatter2015,
  author    = {Vatter, Vincent},
  title     = {Permutation Classes},
  booktitle = {Handbook of Enumerative Combinatorics},
  editor    = {B{\'o}na, Mikl{\'o}s},
  publisher = {CRC Press},
  address   = {Boca Raton},
  pages     = {754--833},
  year      = {2015},
  doi       = {10.1201/b18255}
}

@article{Krattenthaler2001,
  author  = {Krattenthaler, Christian},
  title   = {Permutations with restricted patterns and {Dyck} paths},
  journal = {Advances in Applied Mathematics},
  volume  = {27},
  pages   = {510--530},
  year    = {2001}
}

@article{Bona2012Surprising,
  author  = {B{\'o}na, Mikl{\'o}s},
  title   = {Surprising symmetries in objects counted by {Catalan} numbers},
  journal = {Electronic Journal of Combinatorics},
  volume  = {19},
  number  = {1},
  pages   = {P62},
  year    = {2012}
}

\end{document}